\documentclass[11pt,amsfonts]{article}
\usepackage{graphicx}
\usepackage{latexsym}
\usepackage{amssymb}
\usepackage{amsmath}
\usepackage{enumerate}
\usepackage{color}
\usepackage{layout}
\newtheorem{prop}{Proposition}
\newtheorem{lemma}{Lemma}

\newtheorem{theorem}{Theorem}
\newtheorem{remark}{Remark}

\def\real{{\mathord{{\rm I\kern-2.8pt R}}}}        % Fake blackboard bold R.
\def\inte{{\mathord{{\rm I\kern-2.8pt N}}}}

\def\sZZ{{\rm Z\kern-2.8ptem{}Z}}

\def\z{{\mathchoice
		{\sZZ}
		{\sZZ}
		{\rm Z\kern-0.30em{}Z}
		{\rm Z\kern-0.25em{}Z} }}
\def\sQQ{{\kern 0.27em \vrule height1.45ex width0.03em depth0em
		\kern-0.30em \rm Q}}
\def\qu{{\mathchoice
		{\sQQ}
		{\sQQ}
		{\kern 0.225em \vrule height1.05ex width0.025em depth0em \kern-0.25em \rm Q}
		{\kern 0.180em \vrule height0.78ex width0.020em depth0em \kern-0.20em \rm Q}
}}
\def\sCC{{\kern 0.27em \vrule height1.45ex width0.03em depth0em
		\kern-0.30em \rm C}}
\def\complex{{\mathchoice
		{\sCC}
		{\sCC}
		{\kern 0.225em \vrule height1.05ex width0.025em depth0em \kern-0.25em \rm C}
		{\kern 0.180em \vrule height0.78ex width0.020em depth0em \kern-0.20em \rm C}
}}

\newcommand{\ba}{\begin{array}}
	\newcommand{\ea}{\end{array}}
\newcommand{\be}{\begin{equation}}
	\newcommand{\ee}{\end{equation}}
\newcommand{\bea}{\begin{eqnarray}}
	\newcommand{\eea}{\end{eqnarray}}
\newcommand{\beaa}{\begin{eqnarray*}}
	\newcommand{\eeaa}{\end{eqnarray*}}

\def\z{\zeta}

\font\tenmath=msbm10 \font\sevenmath=msbm7 \font\fivemath=msbm5
\newfam\mathfam \textfont\mathfam=\tenmath
\scriptfont\mathfam=\sevenmath \scriptscriptfont\mathfam=\fivemath

\def \={{\buildrel {\rm (law)} \over =}}

\newcommand{\basa}{\begin{assumption}}
	\newcommand{\easa}{\end{assumption}}

\newcommand{\bas}{\begin{assum}}
	\newcommand{\eas}{\end{assum}}

\newcommand{\ignore}[1]{}
\begin{document}
	
	\renewcommand{\thefootnote}{\fnsymbol{footnote}}
	
	\renewcommand{\thefootnote}{\fnsymbol{footnote}}

	\title{Asymptotic normality for a modified quadratic variation of the Hermite process}
	\author{Antoine Ayache and  Ciprian A. Tudor\vspace*{0.2in} \\
		CNRS, Universit\'e de Lille \\
		Laboratoire Paul Painlev\'e UMR 8524\\
		F-59655 Villeneuve d'Ascq, France.\\
		\quad  antoine.ayache@univ-lille.fr\\
		\quad ciprian.tudor@univ-lille.fr\\
		\vspace*{0.1in} }
	
	\maketitle
	
	\begin{abstract}
	We consider  a modified quadratic variation of the Hermite process based on some well-chosen increments of this process.  These special increments have the very useful property to be independent and identically distributed up to asymptotically negligible remainders. We prove that this modified quadratic variation satisfies a Central Limit Theorem and we derive its rate of convergence under the Wasserstein distance via Stein-Malliavin calculus.  As a consequence, we construct, for the first time in the literature related to Hermite processes, a strongly consistent and asymptotically normal estimator for the Hurst parameter.
\end{abstract}

	\vskip0.3cm
	
	{\bf 2010 AMS Classification Numbers:}  60H15, 60H07, 60G35.

	\vskip0.3cm
	

		\section{Introduction}
	Let an arbitrary integer $q\ge 1$. The Hermite process of order $q$ is one of  the most classical examples of a stochastic process belonging to the $q$th Wiener chaos. When $q=1$ it reduces to the well-known Gaussian fractional Brownian motion. The Hermite process becomes non-Gaussian as soon as $q\ge 2$. Yet, similarly to fractional Brownian motion, it is self-similar with stationary correlated increments possessing long memory property.  We refer to the monographs \cite{PiTa-book} or \cite{T} for a detailed presentation of Hermite processes and their applications and many other related topics. 
	
	%The Hermite processes constitutes a class of self-similar stochastic processes with stationary increments and long memory. The Hermite process of order $q\geq 1$ includes the fractional Brownian motion (obtained for $q=1$), which is the only Gaussian process in this class.
	Stochastic analysis of Hermite and related processes has been developed in the last decades. This is due to the fact that nice properties (non-Gaussianity, finiteness of all moments, self-similarity, long memory, and so on) of these processes make them natural candidates to model various phenomena. Also, stochastic analysis of Hermite processes is interesting in its own right since the classical Itô calculus fails to be applicable to them. It allows for a better understanding of properties of non-Gaussian chaotic stochastic processes, and it raises challenging mathematical problems involving Malliavin calculus and Stein method among many other powerful methodologies coming from probability theory and  functional analysis. We refer to the monographs \cite{jan97,NP-book,N}  for a detailed presentation of these topics. 
	
	One of the problems widely studied concerning the Hermite process is the statistical estimation of its Hurst parameter (or self-similarity index), usually denoted by $H$ and which belongs to the open interval $(\frac{1}{2}, 1)$. This parameter determines many characteristics  of the process (the scaling property, the moments, the intensity of the long memory, the regularity of sample paths, etc.)  and therefore its estimation is of utmost importance for applications. A main approach to construct estimators for the Hurst parameter is based on the analysis of quadratic (or higher order) variations of Hermite process (see e.g.  \cite{CTV2,TV,coeur}), but other techniques have been also employed, such as wavelet analysis in \cite{BaTu,CRTT1,CRTT2}, and least-squares-type estimators in \cite{NT}.  While these approaches lead to consistent estimators for $H$, a general fact appearing in the references is that these estimators are not asymptotically normal, their limit distribution being in  general a Rosenblatt random variable (the value at time 1 of the Hermite process of order $q=2$).  Such a non-Gaussian behavior in statistical estimation of Hurst parameter was even identified a long time ago in the very particular Gaussian case ($q=1$) of fractional Brownian motion, when  $H\in (3/4,1)$ and its estimator is derived form usual quadratic variations (see for instance \cite{Taq75,DoMa79}). Yet, it was shown in \cite{IL97} that in this very special Gaussian case one can recover asymptotic normality for the latter estimator by replacing in quadratic variations the 1-order increments of the process by higher order increments. Unfortunately, such a strategy fails to work for all the non-Gaussian Hermite processes, that is as soon as $q\ge 2$ (see e.g. \cite{CTV}).
	
	The purpose of our present article is to propose a new strategy for overcoming the non-Gaussian behavior in statistical estimation of Hurst parameter of any Hermite process including fractional Brownian motion. To this end, we define a so-called modified quadratic variation of Hermite process, which is obtained only through some well-chosen 1-order increments of this process. The idea behind the definition of the modified quadratic variation comes from the recent paper \cite{A} in which some well-chosen increments of the Hermite process, over consecutive dyadic numbers of the same level, were introduced in order to bound from below local oscillations of the process. Roughly speaking, the crucial advantage in using these well-chosen increments is that their dominant parts are independent and identically distributed.
	
	Thanks to the latter crucial advantage, by using  techniques from  the Stein-Malliavin calculus, we show that the modified quadratic variations sequence satisfies a Central Limit Theorem (CLT) and we  determine the rate of convergence, under the Wasserstein distance, associated to this CLT. Then,  by standard arguments, we define an estimator for the Hurst parameter of the Hermite process based on the modified quadratic variation and we prove its strong consistency and asymptotic normality. %with an evaluation of the rate of convergence to the Gaussian distribution under the Wasserstein distance. 
	From the statistical point of view, our estimator needs less empirical data than those defined in e.g. \cite{CTV} or \cite{TV}, since it includes in its expression only a part of the increments of the Hermite process  over consecutive dyadic numbers with the same level of  the interval $[0,1]$.  We believe that our method has potential to be applied to other stochastic processes related to the class of Hermite processes. In order to illustrate this fact, we treat the case of the Ornstein-Uhlenbeck process associated to the Hermite process.  We analyze  the  behavior of its modified quadratic variation and we discuss the estimation of its Hurst parameter. 
	
	The remaining of our article is organized as follows. In Section \ref{sec2}, we recall the definition of the Hermite process of any order $q\ge 1$ and its basic properties. In Section \ref{sec3}, we focus on some "nice" increments of it which are reminiscent of those previously introduced in \cite{A}; we precisely show how these increments of Hermite process can be decomposed in dominant and negligible parts in such a way that the dominant parts be independent and identically distributed, then we determine the asymptotic behavior of the second and fourth moments of the dominant parts. Section \ref{sec4} is the keystone of our article, we define in terms of the "nice" increments of Hermite process the modified quadratic variation $V_N$ and its dominant part  $V_{N,1}$ as well as its two negligible parts $V_{N,2}$ and $V_{N,3}$; then using techniques from the Stein-Malliavin calculus, we prove that the distribution of $V_{N,1}$, and consequently that of $V_N$, converges to a centered Gaussian distribution at a fast rate quantified in terms of the Wasserstein distance.  The goal of Section \ref{sec5} is to show that the estimator $\widehat{H}_{N}$ for the Hurst parameter, derived from $V_N$, is strongly consistent and asymptotically normal.
	%with a fast rate of convergence toward Gaussian {\bf supprimer: rather similar to that of $V_N$}. 
	In Section \ref{sec6}, it is shown that the method developed in the two previous sections can also be employed to obtain asymptotic normality of modified quadratic variation of the Hermite Ornstein-Uhlenbeck process (the solution to the Langevin equation with Hermite noise) as well as a strongly consistent and asymptotically normal estimator for its Hurst parameter. Finally, Section \ref{app} is the Appendix where we included the notions related with Wiener chaos and Malliavin calculus needed in our work. 
	
	\section{Preliminaries}
	\label{sec2}
	
	In this preliminary part, we introduce the Hermite process and we recall its basic properties. Let an arbitrary integer $q\geq 1$ and a real number $H \in ( \frac{1}{2}, 1)$. One denotes by $ ( Z ^{H,q}_{t}, t\geq 0)$ the Hermite process of order $q$ and with self-similarity index (or Hurst parameter) $H$. Using the convention that, for all $(x,a)\in\mathbb{R}^2$, one has $x_{+}^{a}=x^a$ when $x>0$ and $x_{+}^{a}=0$ otherwise, the process $ ( Z ^{H,q}_{t}, t\geq 0)$ can be defined through multiple Wiener integral as
	\begin{equation}
		\label{hermite}
		Z^{H, q}_{t}= c(H,q) \int_{\mathbb{R} ^{q} } \left( \int_{0}^{t}(u-y_{1} )_{+}^{-\left( \frac{1}{2}+ \frac{1-H}{q}\right)}\ldots (u-y_{q} )_{+}^{-\left( \frac{1}{2}+ \frac{1-H}{q}\right)}du\right) dB(y_{1}) \ldots dB(y_{q}),
	\end{equation}
	where $ (B(y), y \in \mathbb{R})$ is a two-sided Wiener process and $c(H,q)$ is a strictly positive normalizing constant chosen such that  $ \mathbf{E} ( Z^{H,q}_{t}) ^{2}= t^{2H}.$  We can also express the random variable $ Z^{H,q}_{t}$ as 
	\begin{equation*}
		Z ^{H,q}_{t}= I_{q} ( L ^{H,q}_{t}),
	\end{equation*} 
	where $ I_{q}$ stands for the multiple stochastic integral of order $q$ with respect to the Brownian motion $B$ and $ L ^{H,q}$ is the kernel of the Hermite process given by 
	\begin{equation}\label{L}
		L^{H,q}_{t}(y_{1},\ldots, y_{q})= c(H,q)\int_{0}^{t}(u-y_{1} )_{+}^{-\left( \frac{1}{2}+ \frac{1-H}{q}\right)}\ldots (u-y_{q} )_{+}^{-\left( \frac{1}{2}+ \frac{1-H}{q}\right)}du,
	\end{equation}	
	for every $y_{1},\ldots, y_{q}\in \mathbb{R}$.
	
	It is well-known that $ L ^{H,q}_{t} $ belongs to $ L^{2} (\mathbb{R}^{q})$ for every $t\geq 0$ and this ensures  that the Hermite process $ Z ^{H,q}$ is well-defined. Moreover, the process $ (Z^{H,q}_{t}, t\geq 0)$ is $H$-self-similar, it has stationary and correlated increments with long memory and its sample paths are, modulo a modification, H\"older continuous of any arbitrary order $\delta \in (0,H)$. Its covariance function reads
	\begin{equation*}
		\mathbf{E} Z ^{H,q}_{t} Z ^{H,q}_{s}= \frac{1}{2} \left( t ^{2H}+ s^{2H}-\vert t-s\vert ^{2H} \right) \mbox{ for every } s,t\geq 0. 
	\end{equation*}
	Moreover, the increments of the Hermite process satisfy, for every $s,t\geq 0$ and $p\geq 1$, 
	\begin{equation}\label{29a-4}
		\mathbf{E} \left| Z^{H,q}_{t}- Z ^{H,q}_{s} \right| ^{p} = \mathbf{E}\left| Z ^{H,q}_{1} \right| ^{p} \vert t-s\vert ^{Hp}.
	\end{equation}
	
	\begin{remark}
		We notice that a different parametrization of the Hermite process is sometimes used in the literature, and in particular in the reference \cite{A}, whose results are used in the sequel. In this reference, the self-similarity index of the Hermite process (which is $H$ in our work) is given by $q(H-1)+1$. 
	\end{remark}

	\section{Some nice increments of the Hermite process}\label{sec3}
	
	In this part, we analyze some  well-chosen increments of the Hermite process. Such increments have already been introduced in the reference \cite{A}. Among all the increments of the Hermite process between two consecutive dyadic numbers of the same level $N$, we choose certain of them with nice properties. Actually, each chosen increment can be decomposed into the sum of a dominant part and of another term  which is asymptotically negligible.
	%with a smaller $L^{2}(\Omega)$-norm. 
	Moreover, the dominant parts of these increments form a sequence of independent and identically distributed random variables.

	Let us start with some notations.  Let us fix $\beta \in (0,1)$. For every $N\geq 1$, we  consider the set 
	\begin{equation}
		\label{ln}
		\mathcal{L}_{N}=\mathbb{N} \cap \left[ 1, \frac{ 2 ^{N}}{\left[2 ^{N ^{\beta}}\right]}\right],
	\end{equation}
	where $[x]$ denotes the integer part of $x\in \mathbb{R}$. Let us also introduce the following set 
	\begin{equation}\label{lng}
		\mathcal{L}_{N, \gamma}=	\mathcal{L}_{N}\cap \left[ 1, \left[ 2 ^{N ^{\gamma}}\right]\right] \mbox{ with } \gamma <\beta. 
	\end{equation}
	It is clear that, for every $N\geq 1$, we have
	\begin{equation*}
		\mathcal{L}_{N, \gamma}\subseteq \mathcal{L}_{N}
	\end{equation*}
	and 
	\begin{equation}
		\label{8aa-7}
		\left| 	\mathcal{L}_{N, \gamma}\right| \leq 2 ^{N ^{\gamma}},
	\end{equation}
	where $\left| \mathcal{L}_{N, \gamma}\right|$ denotes the cardinality of $\mathcal{L}_{N, \gamma}$. Moreover, for  all $N$ large enough, we have 
	\begin{equation}
		\label{8aa-7bis}
		\left| 	\mathcal{L}_{N, \gamma}\right| \geq 2 ^{N ^{\gamma}}-1.
	\end{equation}
	We mention in passing that the roles of the two parameters $\beta \in (0,1)$ and $\gamma\in (0,\beta)$ is discussed in Remark \ref{rem:rolebega} at the end of Section \ref{sec5}.
	
	For $l \in \mathcal{L} _{N}$, we consider the following increment of length $2 ^{-N}$ of the Hermite process
	\begin{equation}
		\label{deltax}
		\Delta Z ^{H,q}_{l, N}= Z ^{H,q}_{ \frac{ l\left[ 2 ^{N ^{\beta}}\right]+1}{2 ^{N}}}- Z ^{H,q}_{ \frac{ l\left[ 2 ^{N ^{\beta}}\right]}{2 ^{N}}}= Z ^{H,q}_{ e_{l,N,\beta}+ 2 ^{-N}}- Z ^{H,q}_{e_{l,N,\beta}},
	\end{equation}
	where we denoted, for $l\in \mathcal{L}_{N}$,
	\begin{equation}\label{e}
		e_{l, N, \beta}= \frac{l\left[ 2 ^{N ^{\beta}}\right]}{2 ^{N}}.
	\end{equation}
	
	\noindent In view of (\ref{hermite}) and (\ref{e}), $\Delta Z ^{H,q}_{l, N}$ can be expressed as
	
	\begin{eqnarray*}
		\Delta Z ^{H,q}_{l, N}&=&c(H,q)  \int_{\mathbb{R}^{q}}dB(y_{1})\ldots dB(y_{q}) 1_{ \left( -\infty, \frac{l \left[ 2 ^{N ^{\beta}}\right]+1}{2^{N}}\right] ^{q} } (y_{1},..., y_{q})\nonumber \\
		&&\int_{ \frac{l\left[ 2 ^{N ^{\beta}}\right]}{2 ^{N}}}  ^{ \frac{l\left[ 2 ^{N ^{\beta}}\right]+1}{2 ^{N}}}(u-y_{1} )_{+}^{-\left( \frac{1}{2}+ \frac{1-H}{q}\right)}\ldots (u-y_{q} )_{+}^{-\left( \frac{1}{2}+ \frac{1-H}{q}\right)}du\nonumber \\
		&=&c(H,q)  \int_{\mathbb{R}^{q}}dB(y_{1})\ldots dB(y_{q}) 1_{ \left( -\infty, e_{l,N,\beta}+2 ^{-N}\right] ^{q} } (y_{1},..., y_{q})\nonumber\\
		&&\int_{ e_{l,N,\beta}}^{e_{l,N,\beta}+2 ^{-N}}(u-y_{1} )_{+}^{-\left( \frac{1}{2}+ \frac{1-H}{q}\right)}\ldots (u-y_{q} )_{+}^{-\left( \frac{1}{2}+ \frac{1-H}{q}\right)}du.\nonumber 
	\end{eqnarray*}
	Then, we decompose $	\Delta Z ^{H,q}_{l, N}$  as follows
	\begin{equation*}
		\Delta Z ^{H,q}_{l, N}=\tilde{\Delta} Z ^{H,q}_{l, N}+ \check{\Delta} Z ^{H,q}_{l, N}.
	\end{equation*}
	Above we used the notation 
	\begin{eqnarray}
		\tilde{\Delta} Z ^{H,q}_{l, N}&=&c(H,q) \int_{\mathbb{R}^{q}}dB(y_{1})\ldots dB(y_{q}) 1_{ \left( \frac{(l-1) \left[ 2 ^{N ^{\beta}}\right]+1}{2^{N}}, \frac{l \left[ 2 ^{N ^{\beta}}\right]+1}{2^{N}}\right] ^{q} } (y_{1},..., y_{q})\nonumber \\
		&&\int_{ \frac{l\left[ 2 ^{N ^{\beta}}\right]}{2 ^{N}}}  ^{ \frac{l\left[ 2 ^{N ^{\beta}}\right]+1}{2 ^{N}}}(u-y_{1} )_{+}^{-\left( \frac{1}{2}+ \frac{1-H}{q}\right)}\ldots (u-y_{q} )_{+}^{-\left( \frac{1}{2}+ \frac{1-H}{q}\right)}du  \nonumber \\
		&=&c(H,q) \int_{\mathbb{R}^{q}}dB(y_{1})\ldots dB(y_{q}) 1_{ \left(  e_{l-1, N,\beta}+ 2 ^{-N}, e_{l,N,\beta}+ 2 ^{-N}\right]^{q}}(y_{1},..,y_{q}) \nonumber \\
		&&\int_{ e_{l, N,\beta}}^{e_{l,N, \beta}+ 2 ^{-N}}(u-y_{1} )_{+}^{-\left( \frac{1}{2}+ \frac{1-H}{q}\right)}\ldots (u-y_{q} )_{+}^{-\left( \frac{1}{2}+ \frac{1-H}{q}\right)}du \label{9aa-1}
	\end{eqnarray}
	and %, if $ \overline{A}$ denotes the complementary set of $A$. 
	\begin{eqnarray}
		\check{\Delta} Z ^{H,q}_{l, N}&=& c(H,q) \int_{\mathbb{R}^{q}}dB(y_{1})\ldots dB(y_{q}) 1_{\overline{ \left( \frac{(l-1) \left[ 2 ^{N ^{\beta}}\right]+1}{2^{N}}, \frac{l \left[ 2 ^{N ^{\beta}}\right]+1}{2^{N}}\right] ^{q}} } (y_{1},..., y_{q})\nonumber \\
		&&\int_{ \frac{l\left[ 2 ^{N ^{\beta}}\right]}{2 ^{N}}}  ^{ \frac{l\left[ 2 ^{N ^{\beta}}\right]+1}{2 ^{N}}}(u-y_{1} )_{+}^{-\left( \frac{1}{2}+ \frac{1-H}{q}\right)}\ldots (u-y_{q} )_{+}^{-\left( \frac{1}{2}+ \frac{1-H}{q}\right)}du\nonumber\\
		&=&  c(H,q) \int_{\mathbb{R}^{q}}dB(y_{1})\ldots dB(y_{q}) 1_{ \overline{ \left(  e_{l-1, N,\beta}+ 2 ^{-N}, e_{l,N,\beta}+ 2 ^{-N}\right]^{q}}}(y_{1},..,y_{q}) \nonumber \\
		&&\int_{ e_{l, N,\beta}}^{e_{l,N, \beta}+ 2 ^{-N}}(u-y_{1} )_{+}^{-\left( \frac{1}{2}+ \frac{1-H}{q}\right)}\ldots (u-y_{q} )_{+}^{-\left( \frac{1}{2}+ \frac{1-H}{q}\right)}du,\label{9aa-11}
	\end{eqnarray}
	where
	\begin{eqnarray*}
		&&\overline{ \left( \frac{(l-1) \left[ 2 ^{N ^{\beta}}\right]+1}{2^{N}}, \frac{l \left[ 2 ^{N ^{\beta}}\right]+1}{2^{N}}\right] ^{q}}\\
		&&=\left(-\infty, \frac{l \left[ 2 ^{N ^{\beta}}\right]+1}{2^{N}}\right] ^{q}\setminus\left( \frac{(l-1) \left[ 2 ^{N ^{\beta}}\right]+1}{2^{N}}, \frac{l \left[ 2 ^{N ^{\beta}}\right]+1}{2^{N}}\right]^{q}.
	\end{eqnarray*}
	We refer to $\tilde{\Delta} Z ^{H,q}_{l, N}$ as the dominant part of the increment (\ref{deltax}), while $\check{\Delta} Z ^{H,q}_{l, N}$ can be viewed as its negligible part. The reason is given by the results in Propositions \ref{pp1} and \ref{pp2} below. 
	
	The following lemma plays an important role in the sequel. It shows that the dominant parts of the increments (\ref{deltax}) are independent between them and identically distributed, while their negligible parts have all the same distribution. 
	
	\begin{lemma}\label{ll1}
		For $N\geq 1$, we have 
		\begin{enumerate}
			\item The random variables $\tilde{\Delta} Z ^{H,q}_{l, N}, l\in \mathcal{L}_{N}$, given by (\ref{9aa-1}), are independent and identically distributed.
			
			\item The random variables $\check{\Delta} Z ^{H,q}_{l, N}, l\in \mathcal{L}_{N}$, given by (\ref{9aa-11}), are  identically distributed.

		\end{enumerate}
	\end{lemma}
	\begin{proof} For point 1., the independence has been proven in \cite{A}. To prove that $\tilde{\Delta} Z ^{H,q}_{l, N}, l\in \mathcal{L}_{N}$, have the same distribution, we write (by $ = ^{(d)}$ we denote the equality in distribution), via (\ref{9aa-1}) and the change of variables $\tilde{u}= u-e_{l,N,\beta}$ and then $\tilde{y}_{i}= y_{i}- e_{l,N,\beta} $ for $i=1,..., q$,
		\begin{eqnarray*}
			\tilde{\Delta} Z ^{H,q}_{l, N}&=& c(H,q) \int_{\mathbb{R}^{q}}dB(y_{1})\ldots dB(y_{q}) 1_{ \left( e_{l-1, N,\beta}+ \frac{1}{2^{N}}, e_{l,N, \beta}+ \frac{1}{ 2 ^{N}}\right]^{q}}(y_{1},..,y_{q})\\
			&& \int_{0} ^{2 ^{-N}}du(u+ e_{l,N,\beta}-y_{1} )_{+}^{-\left( \frac{1}{2}+ \frac{1-H}{q}\right)}\ldots (u+e_{l,N,\beta}-y_{q} )_{+}^{-\left( \frac{1}{2}+ \frac{1-H}{q}\right)}\\
			&=&  c(H,q) \int_{\mathbb{R}^{q}}dB(y_{1}+ e_{l,N,\beta})\ldots dB(y_{q}+ e_{l, N,\beta}) 1_{ \left( \frac{ -\left[ 2 ^{N ^{\beta}}\right]}{2 ^{N}}+2 ^{-N}, 2 ^{-N}\right] ^{q} }(y_{1},..., y_{q})\\
			&& \int_{0} ^{2 ^{-N}}du (u-y_{1} )_{+}^{-\left( \frac{1}{2}+ \frac{1-H}{q}\right)}\ldots (u-y_{q} )_{+}^{-\left( \frac{1}{2}+ \frac{1-H}{q}\right)}\\
			&=^{(d)}&  c(H,q) \int_{\mathbb{R}^{q}}dB(y_{1})\ldots dB(y_{q}) 1_{ \left(  \frac{ -\left[ 2 ^{N ^{\beta}}\right]}{2 ^{N}}+2 ^{-N}, 2 ^{-N}\right] ^{q} }(y_{1},..., y_{q})\\
			&& \int_{0} ^{2 ^{-N}}du (u-y_{1} )_{+}^{-\left( \frac{1}{2}+ \frac{1-H}{q}\right)}\ldots (u-y_{q} )_{+}^{-\left( \frac{1}{2}+ \frac{1-H}{q}\right)}.
		\end{eqnarray*}
		The last equality is a consequence of the fact that the Brownian motion $B$ has stationary increments. A similar argument can be used to prove point 2.   \end{proof}
	% One denotes by $L ^{2}_{S} (\mathbb{R} ^{q})$ the subspace of $L ^{2} (\mathbb{R} ^{q})$ formed by the functions $h$ which are symmetric in their $q$ variables, that is, for any permutation $\pi$ of the set $\{1,2,\ldots, q\}$, one has, for almost all 
	%$(y_1,y_2,\ldots, y_q)\in\mathbb{R}^q$,
	%\begin{equation*}
	%h (y_{1},y_2,\ldots, y_{q})=h(y_{\pi(1)},y_{\pi(2)},\ldots, y_{\pi(q)}).
	%\end{equation*}
	
	We also need another auxiliary lemma. 
	% Before stating it, we mention that if $D$ is an arbitrary Borel subset of the real line $\mathbb{R}$, we denote by $D^q\subseteq\mathbb{R} ^{q}$  the Cartesian product of $D$ q-times with itself while $\overline{D^{q}}=\mathbb{R} ^{q}\setminus D^q$ is the complement of $D^q$ in $\mathbb{R} ^{q}$. {\color{red} A T ON VRAIMENT BESOIN DE CETTE PRECISION? DE TOUTE FACON ON A UTILISE LA NOTATION  $ D ^{q}$  AVANT...}
	
	%We also need the following auxiliary lemma.
	
	\begin{lemma}\label{ll2}
		
		Let $\varphi $ and $\psi$ be two arbitrary functions belonging to $L ^{2}(\mathbb{R} ^{q})$, and let $D$ be an arbitrary Borel subset of the real line $\mathbb{R}$. We denote by $f$ and $g$ the two functions of $L ^{2} (\mathbb{R} ^{q})$ defined as: $f= \varphi 1_{D^q}$ and $g=\psi 1_{ \overline{D^{q}}}$,  where $\overline{D^{q}}=\mathbb{R} ^{q}\setminus D^q$. Then, for every  integer $a\geq 1$, we have
		\begin{equation}
			\label{ll2:eq1}
			\mathbf{E} I_{q}(f) ^{a} I _{q} (g)  =0.
		\end{equation} 
	\end{lemma}
	\begin{proof} We can approximate  $g$ by simple functions belonging to $L ^{2} (\mathbb{R} ^{q})$ with supports included in $\overline{D^{q}}$. Thus, there exists a sequence $(g_{k}, k\geq 1)$ of such functions for which we have  $g_{k} \to _{k\to \infty} g$ in $ L ^{2} (\mathbb{R} ^{q})$. Moreover, for each $k\geq 1$, the simple function $g_k$ can be chosen in such a way that it can be expressed, for some  strictly positive integer $M_k$, as:
		\begin{equation}\label{29a-2}
			g_{k}= \sum_{ j_{1}^{(k)},..., j_{q}^{(k)}  =1}^{M_{k}} a_{j_{1}^{(k)},..., j_{q}^{(k)}}^{(k)} 1 _{ A_{j_{1}^{(k)}}^{(k)} \times \ldots \times A _{j_{q}^{(k)}}^{(k)}},
		\end{equation}
		where: 
		\begin{itemize}
			\item $A_{1}^{(k)},..., A_{M_{k}}^{(k)}$ are disjoint bounded Borel subsets of $\mathbb{R}$ such that each one of them is included in $D$, or included in $\overline{D}=\mathbb{R}\setminus D$, while $A_{j_{1}^{(k)}}^{(k)} \times....\times A _{j_{q}^{(k)}}^{(k)} \subseteq \overline{D ^{q}}=\mathbb{R}^q\setminus D^q$;
			\item the coefficients $a_{j_{1}^{(k)},..., j_{q}^{(k)}}^{(k)}$, $(j_{1}^{(k)},..., j_{q}^{(k)})\in\{1,\ldots, M_k\}^{q}$, are real numbers such that any one of them vanishes as soon as at least two of its indices $j_{1}^{(k)},..., j_{q}^{(k)}$ coincide i.e. when one has  $j_{m}^{(k)}= j_{n}^{(k)}$ for some $m,n\in \{1,..., q\}$ with  $m\ne n$. 
		\end{itemize}
		Then, we can derive from (\ref{29a-2}) and an elementary property of multiple Wiener integral,
		that 
		\begin{equation}\label{29a-2:bis}
			I_{q} (g_{k})=  \sum_{ j_{1}^{(k)},..., j_{q}^{(k)}  =1}^{M_{k}} a_{j_{1}^{(k)},..., j_{q}^{(k)}}^{(k)} W(A_{j_{1}^{(k)}})\ldots W(A_{j_{q}^{(k)}}), \hskip0.3cm k\geq 1,
		\end{equation}
		where $W$ denotes the Brownian random measure associated with the Brownian motion $B$. Next, observe that, in view of the isometry property of multiple Wiener integral (see (\ref{iso})) and Cauchy-Schwarz inequality, it turns out that, for deriving (\ref{ll2:eq1}), it is enough to show that, for any integers $a\geq 1$ and $ k\geq 1$, one has 
		\begin{equation}
			\label{3o-1}
			\mathbf{E}I_{q}(f) ^{a}I _{q} (g_{k})=0.
		\end{equation}
		By (\ref{29a-2:bis}) and linearity of expectation operator, to get (\ref{3o-1}) it suffices to prove that 
		\begin{equation}
			\label{3o-2}
			\mathbf{E} I_{q}(f) ^{a} W (A_{1})\ldots W(A_{q}) =0,
		\end{equation}
		where $A_{1},\ldots ,A_{q}$ are arbitrary disjoint bounded Borel subsets of $\mathbb{R}$ such that each one of them is included in $D$, or included in $\overline{D}$, while we have
		\begin{equation}
			\label{3o-3}
			A_{1} \times \ldots \times A_{q} \subseteq \overline{D^{q}}.
		\end{equation}
		Observe that (\ref{3o-3}) implies that there exists at least one index $j_{0}\in \{1,2,..., q\}$ such that $ A_{j_{0}}\subseteq \overline{D}$.  Let then $J$ be the nonempty set of indices defined as $J =\{ j\in \{ 1,2,..., q \}, A_{j}\subseteq \overline{D} \}$, and $\overline{J}= \{1,2,...,q\}\setminus J=\{ j\in \{ 1,2,..., q \}, A_{j}\subseteq D \}$ the complement of $J$. We denote by $X_1$ and $X_2$ the two random variables defined as $X_1=I_{q}(f) ^{a} \prod _{j\in \overline{J}}W(A_{j})$ and $X_2=\prod_{j\in J} W (A_{j})$, with the convention that $X_1=I_{q}(f) ^{a}$ when $\overline{J}$ is empty. We denote by ${\cal F}_{\,\overline{D}}$ (resp. ${\cal F}_{D}$) the sigma-algebra generated by all the random variables of the form $W(A)$, where $A$ is an arbitrary bounded Borel subset of $\overline{D}$ (resp. $D$). We know from a very classical property of the Brownian measure $W$ and from the fact that the sets $\overline{D}$ and $D$ are disjoint, that the two sigma-algebras ${\cal F}_{\,\overline{D}}$ and ${\cal F}_{D}$ are independent. Moreover, we know from the definitions of $X_1$ and $f$ and from Lemma 1.2.5 in \cite{N}, that $X_1$ is ${\cal F}_{D}$ measurable. Also, we know from the definition of $X_2$ that $X_2$ is ${\cal F}_{\,\overline{D}}$-measurable. Then, we can derive from the independence of the two sigma-algebras ${\cal F}_{\,\overline{D}}$ and ${\cal F}_{D}$ that $X_1$ and $X_2$ are two independent random variables. Thus, using their definitions we get that 
		\begin{equation}
			\label{3o-2:ter}
			\mathbf{E} I_{q}(f) ^{a} W (A_{1})\ldots W(A_{q}) = \mathbf{E}(X_1X_2)=\mathbf{E}(X_1) \mathbf{E}(X_2)=0,
		\end{equation}
		where the last equality is due to the fact that $\mathbf{E}(X_2)=0$, which can be derived from the independence of the centered Gaussian random variables $W (A_{j})$, $j\in J$, the sets $A_j$ being disjoint.
		
		Finally, it clearly follows from (\ref{3o-2:ter}) that (\ref{3o-2}) is satisfied.
	\end{proof}
	%It follows that the random variable $\prod_{j\in J} W (A_{j}) $ is  independent of $\prod _{j\in \overline{J}} W(A_{j})$ (since $A_{j}, j=1,..., q$ are disjoint sets)  and it is measurable with respect ot the sigma-algebra generated by $\left(W(A), A \subset \overline{D}\right).$ On the other hand, by Lemma 1.2.5 in \cite{N}, the random variable $I_{q}(f)$ is measurable with respect to the sigma-algebra generated by $ \left( W(A), A\subset D\right)$.  
	%
	%Then, for any integers $a\geq 1$ and $k\geq 1$, we can write 
	%\begin{eqnarray*}
	%		\mathbf{E} I_{q}(f) ^{a} W (A_{1})\ldots W(A_{q}) &=& \mathbf{E} \left(I_{q}(f) ^{a} \prod _{j\in \overline{J}} W(A_{j}) \prod_{j\in J} W (A_{j}) \right) \\
	%		&=&  \mathbf{E}\left( I_{q}(f) ^{a} \prod _{j\in \overline{J}} W(A_{j}) \right)\mathbf{E} \prod_{j\in J} W (A_{j})=0,
	%\end{eqnarray*}
	%where the last equality is due to the fact that $\mathbf{E} \prod_{j\in J} W (A_{j})=0.$ So, we obtain (\ref{3o-1}). The conclusion follows by taking the limit as $k\to \infty$ in (\ref{3o-1}). \qed 

	Let us now recall the following result from \cite{A}. 
	
	\begin{prop}\label{pp1}
		For every $N\geq 1$ large enough and $ l \in \mathcal{L}_{N}$, 
		\begin{equation}
			\label{8aa-6}
			\mathbf{E} \left( 	\check{\Delta} Z ^{H,q}_{l, N}\right) ^{2} \leq C 2 ^{-2NH} 2 ^{N ^{\beta}\frac{2H-2}{q}}.
		\end{equation}
	\end{prop}
	
	Below, we deduce the asymptotic behavior of the second and fourth moments of the dominant parts denoted by  $\tilde{\Delta}Z ^{H,q}_{l, N} $. 
	\begin{prop}\label{pp2}
		For all $N\geq 1$ large enough and for every $l\in \mathcal{L}_{N}$,  we have 
		\begin{equation}
			\label{8aa-1}
			\left| 2 ^{2HN} \mathbf{E} \left(\tilde{\Delta}Z ^{H,q}_{l, N} \right) ^{2}-1\right| \leq  C 2 ^{N ^{\beta}\frac{2H-2}{q}},
		\end{equation}
		which implies that
		%and in particular,
		\begin{equation*}
			2 ^{2NH} \mathbf{E} \left(\tilde{\Delta}Z ^{H,q}_{l, N} \right) ^{2} \to _{N\to \infty} 1.
		\end{equation*}
		Also, we have
		\begin{equation}\label{8aa-2}
			\left| 	2 ^{4NH} \mathbf{E} \left(\tilde{\Delta}Z ^{H,q}_{l, N} \right) ^{4}- \mathbf{E}\vert Z ^{H,q}_{1} \vert ^{4} \right| \leq C 2 ^{N ^{\beta}\frac{2H-2}{q}},
		\end{equation}
		which entails that
		%and in particular 
		\begin{equation*}
			2 ^{4NH} \mathbf{E} \left(\tilde{\Delta}Z ^{H,q}_{l, N} \right) ^{4} \to _{N\to \infty} \mathbf{E} \vert Z ^{H,q}_{1} \vert ^{4}. 
		\end{equation*}
	\end{prop}
	\begin{proof}To prove (\ref{8aa-1}), we notice that for $N\geq 1$ and $l\in \mathcal{L}_{N}$,  by (\ref{29a-4}),
		\begin{equation*}
			\mathbf{E} \vert \Delta Z ^{H,q} _{l, N} \vert ^{2}= 2 ^{-2NH}
		\end{equation*}
		and by Lemma \ref{ll2}, 
		\begin{eqnarray}
			\mathbf{E}\vert \Delta Z ^{H,q} _{l, N} \vert ^{2}&=&\mathbf{E} \vert \tilde{\Delta} Z ^{H,q}_{l, N} \vert ^{2}+ \mathbf{E} \vert \check{\Delta} Z ^{H,q}_{l, N} \vert ^{2}+ 2\mathbf{E}  \tilde{\Delta} Z ^{H,q}_{l, N} \check{\Delta} Z ^{H,q}_{l, N}\nonumber \\
			&=&\mathbf{E} \vert \tilde{\Delta} Z ^{H,q}_{l, N} \vert ^{2}+ \mathbf{E} \vert \check{\Delta} Z ^{H,q}_{l, N} \vert ^{2}.\label{8aa-8}
		\end{eqnarray}
		On the other hand, by Proposition \ref{pp1},
		\begin{equation*}
			2 ^{2NH}  \mathbf{E}  \left(\check{\Delta}Z ^{H,q}_{l, N} \right) ^{2}\leq C 2 ^{N ^{\beta}\frac{2H-2}{q}}\to _{N \to \infty}0. 
		\end{equation*}
		Relation (\ref{8aa-8}) implies
		\begin{eqnarray*}
			2 ^{2HN} \mathbf{E} \left(\tilde{\Delta}Z ^{H,q}_{l, N} \right) ^{2}-1&=& \left( 2 ^{2HN} \mathbf{E} \vert \Delta Z ^{H,q} _{l, N}\vert ^{2}-1 \right) -  2 ^{2HN}  \mathbf{E} \left(\check{\Delta}Z ^{H,q}_{l, N} \right) ^{2}\\
			&=&- 2 ^{2HN} \mathbf{E} \left(\check{\Delta}Z ^{H,q}_{l, N} \right) ^{2}
		\end{eqnarray*}
		hence
		\begin{equation*}
			\left| 2 ^{2HN} \mathbf{E} \left(\tilde{\Delta}Z ^{H,q}_{l, N} \right) ^{2}-1\right| \leq C 2 ^{N ^{\beta}\frac{2H-2}{q}}.
		\end{equation*}
		So, (\ref{8aa-1}) is obtained.  Let us now show (\ref{8aa-2}).
		
		By (\ref{29a-4}), we have
		\begin{equation*}
			\mathbf{E} \left(\Delta Z ^{H,q} _{l, N} \right) ^{4}= 2 ^{-4HN} \mathbf{E} \vert Z ^{H,q}_{1}\vert ^{4},
		\end{equation*}
		for every $ l \in \mathcal{L}_{N}$. On the other hand, by Lemma \ref{ll2} 
		\begin{eqnarray*}
			\mathbf{E} \left(\Delta Z ^{H,q}_{l, N} \right) ^{4}&=& \mathbf{E} \left( \tilde{\Delta} Z ^{H,q}_{l, N}\right) ^{4} + \mathbf{E}\left( \check{\Delta}Z ^{H,q}_{l, N}\right) ^{4} \\
			&&+ 6 \mathbf{E} \left(  \tilde{\Delta}Z ^{H,q}_{l, N}\right) ^{2}   \left(  \check{\Delta}Z ^{H,q}_{l, N}\right) ^{2}+   4\mathbf{E}  \tilde{\Delta}Z ^{H,q}_{l, N} \left(  \check{\Delta}Z ^{H,q}_{l, N}\right) ^{3}. 
		\end{eqnarray*}
		Using the hypercontractivity property (\ref{hyper}) and Proposition \ref{pp1}, we obtains that 
		\begin{eqnarray}\label{8aa-3}
			\mathbf{E}\left( \check{\Delta}Z ^{H,q}_{l, N}\right) ^{4} &\leq & C \left( \mathbf{E}\left( \check{\Delta}Z ^{H,q}_{l, N}\right) ^{2}\right) ^{2} \leq C 2 ^{-4HN}2 ^{N ^{\beta} \frac{4H-4}{q}}.
		\end{eqnarray}
		On the other hand, Cauchy-Schwarz's inequality, (\ref{hyper}) and Proposition \ref{pp1} imply that
		\begin{eqnarray}\label{8aa-4}
			&& \mathbf{E}\left(  \tilde{\Delta}Z ^{H,q}_{l, N}\right) ^{2}   \left(  \check{\Delta}Z ^{H,q}_{l, N}\right) ^{2}\leq        \left(  \mathbf{E} \left(  \tilde{\Delta}Z ^{H,q}_{l, N}\right) ^{4}   \right) ^{\frac{1}{2}}     \left(  \mathbf{E}\left(  \check{\Delta}Z ^{H,q}_{l, N}\right) ^{4}   \right) ^{\frac{1}{2}}    \nonumber\\
			&&\leq     C     \mathbf{E} \left(  \tilde{\Delta}Z ^{H,q}_{l, N}\right) ^{2}  \mathbf{E} \left(   \check{\Delta}Z ^{H,q}_{l, N}\right) ^{2}\leq C    2 ^{-4HN}2 ^{N ^{\beta} \frac{2H-2}{q}}
		\end{eqnarray}
		and
		\begin{eqnarray}
			&&\mathbf{E}  \tilde{\Delta}Z ^{H,q}_{l, N} \left(  \check{\Delta}Z ^{H,q}_{l, N}\right) ^{3}\leq  \left(  \mathbf{E} \left(  \tilde{\Delta}Z ^{H,q}_{l, N}\right) ^{2}   \right) ^{\frac{1}{2}}     \left(  \mathbf{E} \left(  \check{\Delta}Z ^{H,q}_{l, N}\right) ^{6}   \right) ^{\frac{1}{2}}    \nonumber\\
			&&\leq C \left(  \mathbf{E} \left(  \tilde{\Delta}Z ^{H,q}_{l, N}\right) ^{2}   \right) ^{\frac{1}{2}}    \left(  \mathbf{E} \left(  \check{\Delta}Z ^{H,q}_{l, N}\right) ^{2}   \right) ^{\frac{3}{2}}   \leq  C 2 ^{-4HN}2 ^{N ^{\beta} \frac{3H-3}{q}}.\label{3o-4}
		\end{eqnarray}
		Finally, combining (\ref{8aa-3}), (\ref{8aa-4}) and (\ref{3o-4}), we get (\ref{8aa-2}). \end{proof}

	\section{Modified quadratic variation of the Hermite process}\label{sec4}
	
	We define the (centered) modified quadratic variation of the Hermite process, constructed by using not all the increments of the Hermite process over the consecutive dyadic numbers of arbitrary level $N$ of the unit interval $[0,1]$, but only the increments  with nice properties introduced in Section \ref{sec3}.  For any integer $N\geq 1$, let us set
	\begin{eqnarray}
		V_{N} &=& \frac{ 2 ^{2HN}}{ \sqrt{ \vert \mathcal{L}_{N, \gamma}\vert}}\sum_{l\in \mathcal{L}_{N, \gamma}} \left[ \left( Z ^{H,q}_{ e_{l,N, \beta}+2 ^{-N}}-  Z ^{H,q}_{ e_{l,N, \beta}}\right) ^{2}-   \mathbf{E}\left( Z ^{H,q}_{ e_{l,N, \beta}+2 ^{-N}}-  Z ^{H,q}_{ e_{l,N, \beta}}\right) ^{2}\right]\nonumber \\
		&=&\frac{ 2 ^{2HN}}{ \sqrt{ \vert \mathcal{L}_{N, \gamma}\vert}}\sum_{l\in \mathcal{L}_{N, \gamma}} \left[( \Delta Z^{H,q}_{l,N}) ^{2}- \mathbf{E}( \Delta Z^{H,q}_{l,N} )^{2}\right],\label{vn}
	\end{eqnarray}
	where $\vert \mathcal{L}_{N, \gamma}\vert$ is the cardinality of the set $ \mathcal{L}_{N, \gamma}$ defined in (\ref{lng}), and $ \Delta Z^{H,q}_{l,N}, e_{l,N, \beta}$  are given by (\ref{deltax}) and (\ref{e}), respectively.
	
	We will show that the sequence $( V_{N}, N\geq 1)$ satisfies a Central Limit Theorem (CLT) and we will deduce the rate of convergence under the Wasserstein distance. To do this, we decompose $V_{N}$ as follows: 
	\begin{equation}\label{deco}
		V_{N}= V_{N,1}+ V_{N, 2}+ V_{N,3},
	\end{equation}
	where
	\begin{equation}
		\label{vn1}
		V_{N,1}= \frac{ 2 ^{2HN}}{ \sqrt{ \vert \mathcal{L}_{N, \gamma}\vert}}\sum_{l\in \mathcal{L}_{N, \gamma}}  \left[( \tilde{\Delta} Z^{H,q}_{l,N}) ^{2}- \mathbf{E}( \tilde{\Delta} Z^{H,q}_{l,N} )^{2}\right],
	\end{equation}
	
	\begin{equation}
		\label{vn2}
		V_{N,2}= \frac{ 2 ^{2HN}}{ \sqrt{ \vert \mathcal{L}_{N, \gamma}\vert}}\sum_{l\in \mathcal{L}_{N, \gamma}}  \left[( \check{\Delta} Z^{H,q}_{l,N}) ^{2}- \mathbf{E}( \check{\Delta} Z^{H,q}_{l,N} )^{2}\right],
	\end{equation}
	and
	\begin{equation}
		\label{vn3}
		V_{N,3}= 2 \frac{ 2 ^{2HN}}{ \sqrt{ \vert \mathcal{L}_{N, \gamma}\vert}}\sum_{l\in \mathcal{L}_{N, \gamma}}   \tilde{\Delta} Z^{H,q}_{l,N} \check{\Delta} Z^{H,q}_{l,N}.
	\end{equation}
	The strategy to derive the CLT for $V_{N}$ is the following: we first prove that the two last terms in the right-hand side of the equality (\ref{deco}), namely $V_{N.2}$ and $ V_{N,3}$, converge to zero in $ L^{1} (\Omega)$ as $ N\to \infty$ (these two terms will be called the remainders), then using tools from the Stein-Malliavin calculus,  we show that the sequence $V_{N,1}$ converges in distribution, as $ N\to \infty$, to a Gaussian random variable. 
	Let us first deal with the remainders.

	\subsection{The remainders}
	The goal of this subsection is to  show that $V_{N,2}$ and $V_{N,3}$ converge to zero in $ L^{1} (\Omega)$ as $N\to \infty$.

	\begin{prop}\label{pp3}
		Consider the sequence $(V_{N,2}, N\geq 1)$ given by (\ref{vn2}). Then, for $N$ sufficiently large,
		\begin{equation*}
			\mathbf{E}\vert V_{N,2} \vert \leq C  2 ^{\frac{ N ^{\gamma}}{2}+ N ^{\beta} \frac{2H-2}{q} }.
		\end{equation*}
		In particular, $V_{N,2}$ converges to zero in $ L^{1}(\Omega)$ as $N\to \infty$. 
	\end{prop}
	\begin{proof} By using the inequalities (\ref{8aa-6})  and (\ref{8aa-7}), we obtain that 
		\begin{eqnarray*}
			\mathbf{E}\vert V_{N,2}\vert &\leq & C  \frac{ 2 ^{2HN}}{ \sqrt{ \vert \mathcal{L}_{N, \gamma}\vert}}\sum_{l\in \mathcal{L}_{N, \gamma}}\mathbf{E}( \check{\Delta} Z^{H,q}_{l,N}) ^{2}\leq C \sqrt{ \vert \mathcal{L}_{N, \gamma}\vert}2 ^{ N ^{\beta}\frac{2H-2}{q}}\\
			&\leq &C 2 ^{\frac{ N ^{\gamma}}{2}+ N ^{\beta} \frac{2H-2}{q} }.
		\end{eqnarray*}
		Since $\gamma <\beta$ (see the assumption (\ref{lng})), we deduce that $ \mathbf{E}\vert V_{N,2} \vert \to _{N\to \infty}0$.   \end{proof}

	Regarding the summand $V_{N,3}$, we have the following result. 
	\begin{prop}\label{pp4}
		Consider the sequence $(V_{N,3}, N\geq 1)$ given by (\ref{vn3}). Then, for $N$ large enough, 
		\begin{equation*}
			\mathbf{E}\vert V_{N,3} \vert \leq C  2 ^{\frac{ N ^{\gamma}}{2}+ N ^{\beta} \frac{H-1}{q} }.
		\end{equation*}
		In particular, $V_{N,3}$ converges to zero in $ L^{1}(\Omega)$ as $N\to \infty$. 
	\end{prop}
	\begin{proof}By (\ref{8aa-8}), we clearly have for $l\in \mathcal{L}_{N}$,
		\begin{eqnarray*}
			\mathbf{E} \left( \tilde{\Delta} Z ^{H,q}_{l, N} \right) ^{2} \leq   2 ^{-2HN}.
		\end{eqnarray*}
		By the above inequality, (\ref{8aa-7}), Cauchy-Schwarz's inequality and (\ref{8aa-6}), we get 
		\begin{eqnarray*}
			\mathbf{E} \vert V_{N,3}\vert &\leq &  \frac{ 2 ^{2HN}}{ \sqrt{ \vert \mathcal{L}_{N, \gamma}\vert}}\sum_{l\in \mathcal{L}_{N, \gamma}} \left( 	\mathbf{E} \left( \tilde{\Delta} Z ^{H,q}_{l, N} \right) ^{2}\right) ^{\frac{1}{2}} \left( 	\mathbf{E} \left( \check{\Delta} Z ^{H,q}_{l, N} \right) ^{2}\right) ^{\frac{1}{2}}\\
			&\leq & C \sqrt{ \vert \mathcal{L}_{N, \gamma}\vert} 2 ^{ N ^{\beta}\frac{H-1}{q}}\leq C 2 ^{\frac{ N ^{\gamma}}{2}+ N ^{\beta} \frac{H-1}{q} }.
		\end{eqnarray*} 
		Again, the  assumption $\gamma <\beta$ (see (\ref{lng})) ensures that $ \mathbf{E} \vert V_{N,3}\vert  \to _{ N \to \infty} 0$. \end{proof}

	\subsection{The main term: Central Limit Theorem}
	%Let us now analyze the sequence $(V_{N,1}, N\geq 1)$ given by (\ref{vn1}). 
	The goal of this subsection is to show that the sequence $(V_{N,1}, N\geq 1)$, given by (\ref{vn1}), converges in distribution to a centered Gaussian variable, and to estimate the rate of convergence for the Wasserstein distance via Stein-Malliavin calculus. 
	
	First, let us calculate the asymptotic variance of $V_{N,1}$. Notice that ${\rm Var}(V_{N,1})=\mathbf{E} \vert V_{N,1} \vert ^{2}$ since $V_{N,1}$ is centered. We have, for every $N\geq 1$, 
	\begin{eqnarray*}
		\mathbf{E} \vert V_{N,1} \vert ^{2} &=& \frac{ 2 ^{4HN}}{\vert \mathcal{L}_{ N, \gamma}\vert}  \sum_{k,l\in \mathcal{L}_{N, \gamma}}\mathbf{E}\left[ \left( \left( \tilde{\Delta} Z ^{H,q}_{l, N} \right) ^{2}-\mathbf{E}\left( \tilde{\Delta} Z ^{H,q}_{l, N} \right) ^{2}\right) \left( \left( \tilde{\Delta} Z ^{H,q}_{k, N} \right) ^{2}-\mathbf{E}\left( \tilde{\Delta} Z ^{H,q}_{k, N} \right) ^{2}\right) \right]\\
		&=&  \frac{ 2 ^{4HN}}{\vert \mathcal{L}_{ N, \gamma}\vert} \sum_{l\in \mathcal{L}_{N, \gamma}}\mathbf{E} \left( \left( \tilde{\Delta} Z ^{H,q}_{l, N} \right) ^{2}-\mathbf{E}\left( \tilde{\Delta} Z ^{H,q}_{l, N} \right) ^{2}\right)^{2}\\
		&=& \frac{ 2 ^{4HN}}{\vert \mathcal{L}_{ N, \gamma}\vert}  \sum_{l\in \mathcal{L}_{N, \gamma}}\left[ \mathbf{E} \left( \tilde{\Delta} Z ^{H,q}_{l, N} \right) ^{4}- 
		\left( \mathbf{E} \left( \tilde{\Delta} Z ^{H,q}_{l, N} \right) ^{2}\right) ^{2} \right] \\
		&=& 2 ^{4HN} \left[ \mathbf{E} \left( \tilde{\Delta} Z ^{H,q}_{l_{0}, N} \right) ^{4}- 
		\left( \mathbf{E} \left( \tilde{\Delta} Z ^{H,q}_{l_{0}, N} \right) ^{2}\right) ^{2} \right],
	\end{eqnarray*}
	for some $l_{0}\in \mathcal{L}_{ N, \gamma}$ (for instance one may take $l_0=1$). Notice that, in our previous computations the second and the  {\color{red}} fourth equalities result from Lemma \ref{ll1}, point 1. It follows from these computations and Proposition \ref{pp2} that for $N$ large enough, 
	\begin{equation}\label{9aa-7}
		\left| 	\mathbf{E}  \vert V_{N,1} \vert ^{2} -\left(  \mathbf{E} \vert Z ^{H,q}_{1} \vert ^{4}-1\right) \right| \leq C 2 ^{N ^{\beta}\frac{2H-2}{q}},
	\end{equation}
	and in particular,
	
	\begin{eqnarray*}
		\mathbf{E}  \vert V_{N,1} \vert ^{2} \to _{N\to \infty} \mathbf{E} \vert Z ^{H,q}_{1} \vert ^{4}-1.
	\end{eqnarray*}

	For later purpose, we need the following remark. 
	
	\begin{remark}
		\label{rem:bvn}
		It follows from (\ref{deco}), Propositions \ref{pp3} and \ref{pp4}, (\ref{9aa-7}) and Cauchy-Schwarz's inequality that $\sup_{N\in\mathbb{N}} 	\mathbf{E}  \vert V_{N} \vert<\infty$. Then using the fact that the random variable $V_N$ belongs to a finite sum of Wiener chaoses of orders not more than $2q$ (the latter fact results from (\ref{vn}), (\ref{hermite}) and the product formula (\ref{prod})) and Theorem 5.10 in \cite{jan97}, we get that $\sup_{N\in\mathbb{N}} 	\mathbf{E}  \vert V_{N} \vert^p<\infty$, for any fixed strictly positive real number $p$.  
	\end{remark} 
	
	Now, we need to define some distances between the probability distributions of random variables. We refer to \cite{NP-book}, Appendix C, for more details. Usually, the  distance between the laws of two real-valued random variables $F$ and $G$ is defined as
	\begin{equation}
		\label{dw}
		d_{W} (F, G)= \sup_{h\in \mathcal{A}}\left| \mathbf{E}h(F)-\mathbf{E}h(G)\right|,
	\end{equation}
	where $\mathcal{A}$ is a class of functions satisfying $h(F), h(G)\in L ^{1} (\Omega)$ for every $h\in \mathcal{A}$. When $\mathcal{A}$ is the set of Lipschitz continuous functions $h:\mathbb{R} \to \mathbb{R}$ such that $\Vert h\Vert _{Lip} \leq 1$, where
	$$\Vert h\Vert _{Lip}= \sup_{x, y\in \mathbb{R} , x\not=y} \frac{ \vert h(x)-h(y)\vert} {\vert x-y\vert },$$
	then (\ref{dw}) gives the Wasserstein distance. When $\mathcal{A}$ is the set of indicator functions $\{ 1_{-\infty, z]}, z\in \mathbb{R}\}$ then (\ref{dw}) gives the Kolmogorov distance,  for $\mathcal{A}= \{ 1_{B}, B \in \mathcal{B}(\mathbb{R})\}$ we have the total variation distance, while the choice of $\mathcal{A}$ to be the class of functions $h$ with $\Vert h\Vert _{Lip}+ \Vert h\Vert _{\infty}<\infty$  leads to the Fortet-Mourier distance.

	Let us recall a classical result from Stein-Malliavin calculus (see Theorem 5.1.3 and Remark 5.1.4 in \cite{NP-book}). Below, $d$ could be any of the above distances (Kolmogorov, Total variation, Wasserstein or Fortet-Mourier). We refer to the Appendix for the definition of the Malliavin derivative $D$ and of the Ornstein-Uhlenbeck operator $L$ with respect to an isonormal Gaussian process$(B (\varphi),\varphi\in{\mathcal{H}})$, where  ${\mathcal{H}}$ is a real separable infinite-dimensional Hilbert space. Notice that in our present article, we have $\mathcal{H}= L ^{2} (\mathbb{R})$.

	\begin{theorem}
		\label{tt1}
		Let $F$ be a random variable belonging to a finite sum of Wiener chaoses such that $\mathbf{E}F=0$ and $ \mathbf{E} F ^{2}= \sigma ^{2}$. Let $\sigma^{\prime}>0$. Then
		\begin{equation*}
			d\big(F, N(0, (\sigma^{\prime}) ^{2})\big)\leq C \left( \sqrt{{\rm Var} \left( \langle DF, D(-L)^{-1}F\rangle _{\mathcal{H}} \right) }+ \vert \sigma ^{2}- (\sigma^{\prime}) ^{2} \vert \right),
		\end{equation*}
		where $C>0$ is a universal constant and $N(0, (\sigma^{\prime}) ^{2})$ the centered Gaussian (normal) distribution with standard deviation $\sigma^{\prime}$. 
	\end{theorem}

	By using the above result, we obtain the following limit theorem for the sequence $(V_{N,1}, N\geq 1)$, defined in (\ref{vn1}). By $" \to ^{(d)}"$ we denote the convergence in distribution and all the scalar products appearing in the sequel are in $ L ^{2}(\mathbb{R})$. 
	
	\begin{prop}\label{pp5}
		Consider the sequence $\left(V_{N,1}, N\geq 1\right) $ given by (\ref{vn1}). As $N\to \infty$, 
		\begin{equation*}
			V_{N,1} \to ^{(d)} N\left(0, \left( \mathbf{E} \vert Z^{H,q}_{1}\vert ^{4}-1\right) \right), 
		\end{equation*}
		and for $N$ large enough, 
		\begin{equation*}
			d\left( V_{N,1}, N\left(0, \mathbf{E} \vert Z ^{H,q}_{1} \vert ^{4}-1\right) \right) \leq C(q) 2 ^{-\frac{N^{\gamma}}{2}}.
		\end{equation*}
	\end{prop}
	\begin{proof}In order to use Theorem \ref{tt1}, let us first compute the quantity 
		$$ \langle DV_{N,1}, D(-L) ^{-1} V_{N,1} \rangle_{\mathcal{H}}.$$
		By (\ref{9aa-1}), we can write, for $l\in \mathcal{L}_{N}$,
		\begin{equation}\label{9aa-4}
			\tilde{\Delta}Z ^{H,q}_{l, N}= I_{q} (g_{l, N})
		\end{equation}
		with
		\begin{eqnarray*}
			g_{l, N}(y_{1},..., y_{q})&=& c(H,q)1_{ \left(  e_{l-1, N,\beta}+ 2^{-N}, e_{l,N,\beta}+2 ^{-N}\right]^{q}}(y_{1},..,y_{q}) \nonumber \\
			&&\int_{ e_{l, N,\beta}}^{e_{l,N, \beta}+ 2 ^{-N}}(u-y_{1} )_{+}^{-\left( \frac{1}{2}+ \frac{1-H}{q}\right)}\ldots (u-y_{q} )_{+}^{-\left( \frac{1}{2}+ \frac{1-H}{q}\right)}du,
		\end{eqnarray*}
		for every $y_{1},..., y_{q} \in \mathbb{R}$. By the product formula for multiple stochastic integrals (\ref{prod}) and (\ref{27a-1}),
		\begin{eqnarray*}
			\left( 	\tilde{\Delta}Z ^{H,q}_{l, N}\right) ^{2}- \mathbf{E}	\left( 	\tilde{\Delta}Z ^{H,q}_{l, N}\right) ^{2}= \sum_{r=0} ^{q-1} r!  \binom{q}{r} ^{2} I_{2q-2r} ( g_{l, N} \otimes _{r} g_{l, N})
		\end{eqnarray*}
		and therefore
		\begin{equation}\label{9aa-2}
			D_{\ast} \left[ 	\left( 	\tilde{\Delta}Z ^{H,q}_{l, N}\right) ^{2}-\mathbf{E}	\left( 	\tilde{\Delta}Z ^{H,q}_{k, N}\right) ^{2}\right] =\sum_{r=0} ^{q-1} r!  \binom{q}{r}^{2}(2q-2r) I_{2q-2r-1} ( g_{l, N} \tilde{\otimes} _{r} g_{l, N}(\cdot, \ast))
		\end{equation}
		and
		\begin{equation}\label{9aa-3}
			D _{\ast}(-L) ^{-1}\left[ 	\left( 	\tilde{\Delta}Z ^{H,q}_{l, N}\right) ^{2}-\mathbf{E}	\left( 	\tilde{\Delta}Z ^{H,q}_{k, N}\right) ^{2}\right] =\sum_{r=0} ^{q-1} r! \binom{q}{r}^{2} I_{2q-2r-1} ( g_{l, N} \tilde{\otimes} _{r} g_{l, N}(\cdot, \ast)).
		\end{equation}
		Above, $"\cdot"$ stands for the $2q-2r-1$ variables associated with the integral $I_{2q-2r-1}$, and $ "\ast"$ represents the remaining variable.
		%while $ "\ast"$ represents one variable. 
		The scalar product in $\mathcal{H}=L^{2}(\mathbb{R})$ considered below is with respect to the variable $"\ast"$,  Now,
		\begin{eqnarray*}
			&& \langle DV_{N,1}, D(-L) ^{-1} V_{N,1} \rangle_{\mathcal{H}}\\
			&=& \frac{ 2 ^{4HN}}{ \vert \mathcal{L}_{N, \gamma}\vert}\sum_{l,k\in \mathcal{L}_{ N, \gamma}}\langle 	D_{\ast} \left[ 	\left( 	\tilde{\Delta}Z ^{H,q}_{l, N}\right) ^{2} -	\mathbf{E}\left( 	\tilde{\Delta}Z ^{H,q}_{l, N}\right) ^{2}\right], 	D_{\ast}(-L) ^{-1} \left[ 	\left( 	\tilde{\Delta}Z ^{H,q}_{k, N}\right) ^{2} -\mathbf{E}	\left( 	\tilde{\Delta}Z ^{H,q}_{k, N}\right) ^{2}\right]\rangle_{\mathcal{H}} \\
			&=& \frac{ 2 ^{4HN}}{ \vert \mathcal{L}_{N, \gamma}\vert}\sum_{l\in \mathcal{L}_{ N, \gamma}}\langle 	D_{\ast} \left[ 	\left( 	\tilde{\Delta}Z ^{H,q}_{l, N}\right) ^{2} -	\mathbf{E}\left( 	\tilde{\Delta}Z ^{H,q}_{l, N}\right) ^{2}\right], 	D_{\ast}(-L) ^{-1} \left[ 	\left( 	\tilde{\Delta}Z ^{H,q}_{l, N}\right) ^{2} -\mathbf{E}	\left( 	\tilde{\Delta}Z ^{H,q}_{l, N}\right) ^{2}\right]\rangle_{\mathcal{H}}.
		\end{eqnarray*}
		We used the fact that, since $g_{l, N}, g_{k, N}$ have disjoint supports coordinate by coordinate for $l\not=k$ (see also Lemma 4 in \cite{BDT}), then 
		\begin{eqnarray*}
			&&	\langle 	D_{\ast} \left[ 	\left( 	\tilde{\Delta}Z ^{H,q}_{l, N}\right) ^{2} -	\mathbf{E}\left( 	\tilde{\Delta}Z ^{H,q}_{l, N}\right) ^{2}\right], 	D_{\ast}(-L) ^{-1} \left[ 	\left( 	\tilde{\Delta}Z ^{H,q}_{k, N}\right) ^{2} -\mathbf{E}	\left( 	\tilde{\Delta}Z ^{H,q}_{k, N}\right) ^{2}\right]\rangle_{\mathcal{H}}\\
			&&=\int_{\mathbb{R}} 	D_{x} \left[ 	\left( 	\tilde{\Delta}Z ^{H,q}_{l, N}\right) ^{2} -	\mathbf{E}\left( 	\tilde{\Delta}Z ^{H,q}_{l, N}\right) ^{2}\right]  	D_{x}(-L) ^{-1} \left[ 	\left( 	\tilde{\Delta}Z ^{H,q}_{l, N}\right) ^{2} -\mathbf{E}	\left( 	\tilde{\Delta}Z ^{H,q}_{l, N}\right) ^{2}\right]dx =0.
		\end{eqnarray*}
		By (\ref{9aa-2}),  (\ref{9aa-3}) and the definition of the contraction (\ref{contra}), 
		
		\begin{eqnarray*}
			&& \langle DV_{N,1}, D(-L) ^{-1} V_{N,1} \rangle_{\mathcal{H}}\\
			&=& \frac{ 2 ^{4HN}}{ \vert \mathcal{L}_{N, \gamma}\vert}\sum_{l\in \mathcal{L}_{ N, \gamma}}\sum_{r_{1}, r_{2}=0}^{q-1} r_{1}! r_{2} ! \binom{q}{r_{1}}^{2}\binom{q}{r_{2}}^{2} (2q-2r_{1}) \\
			&& \int_{\mathbb{R}} dx I_{2q-2r_{1}-1}\left( g_{l,N}\tilde{\otimes} _{r_{1}} g_{l, N} (\cdot, x) \right) I_{2q-2r_{2}-1}\left( g_{l,N}\tilde{\otimes} _{r_{2}} g_{l, N} (\cdot, x) \right) \\
			&=&  \frac{ 2 ^{4HN}}{ \vert \mathcal{L}_{N, \gamma}\vert}\sum_{l\in \mathcal{L}_{ N, \gamma}}\sum_{r_{1}, r_{2}=0}^{q-1} r_{1}! r_{2} ! \binom{q}{r_{1}}^{2}\binom{q}{r_{2}}^{2} (2q-2r_{1}) \\
			&& \sum_{a=0} ^{(2q-2r_{1}-1)\wedge (2q-2r_{2}-1)} a! \binom{2q-2r_{1}-1}{a} \binom{2q-2r_{2}-1}{a}\\
			&&\times  I _{4q-2r_{1}-2r_{2}-2a-2} \left( (g_{l, N}\tilde{\otimes} _{r_{1}}g_{l, N})\otimes _{a+1}(g_{l, N}\tilde{\otimes} _{r_{2}}g_{l, N}) \right),
		\end{eqnarray*}
		where we used again the product formula (\ref{prod}). By separating the terms with $r_{1}= r_{2} $ and $ r_{1} \not=r_{2}$ and by using (\ref{27a-1}), 
		\begin{eqnarray}
			&& \langle DV_{N,1}, D(-L) ^{-1} V_{N,1} \rangle_{\mathcal{H}} = \mathbf{E}  \langle DV_{N,1}, D(-L) ^{-1} V_{N,1} \rangle_{\mathcal{H}} \label{9aa-6} \\
			&&+\frac{ 2 ^{4HN}}{ \vert \mathcal{L}_{N, \gamma}\vert}\sum_{l\in \mathcal{L}_{ N, \gamma}}\sum_{r=0} ^{q-1} r! ^{2} \binom{q}{r}^{4} (2q-2r) \sum_{a=0} ^{2q-2r-2}
			a!\binom{2q-2r-1}{a}^{2} \nonumber \\
			&& I_{4q-4r-2a-2}\left( (g_{l, N}\tilde{\otimes} _{r}g_{l, N})\otimes _{a+1}(g_{l, N}\tilde{\otimes} _{r}g_{l, N}) \right)\nonumber \\
			&&+ \frac{ 2 ^{4HN}}{ \vert \mathcal{L}_{N, \gamma}\vert}\sum_{l\in \mathcal{L}_{ N, \gamma}}\sum_{r_{1}, r_{2}=0; r_{1}\not=r_{2}}^{q-1} r_{1}! r_{2} ! \binom{q}{r_{1}}^{2} \binom{q}{r_{2}}^{2} (2q-2r_{1}) \nonumber\\
			&& \sum_{a=0} ^{(2q-2r_{1}-1)\wedge (2q-2r_{2}-1)} a! \binom{2q-2r_{1}-1}{a} \binom{2q-2r_{2}-1}{a}\nonumber\\
			&&\times I _{4q-2r_{1}-2r_{2}-2a-2} \left( (g_{l, N}\tilde{\otimes} _{r_{1}}g_{l, N})\otimes _{a+1}(g_{l, N}\tilde{\otimes} _{r_{2}}g_{l, N}) \right).\nonumber
		\end{eqnarray}
		In the sequel we denote by $\|\cdot\|$ the norm associated with a space $L^2(\mathbb{R}^s)$, for any arbitrary positive integer $s$. It follows from (\ref{9aa-4}) and the isometry property of multiple stochastic integral (see (\ref{iso})) that
		\begin{equation}\label{9aa-5}
			q! \Vert g_{l,N}\Vert ^{2}= \mathbf{E} (\tilde{\Delta} Z ^{H,q}_{l, N})^{2} \leq 2 ^{-2HN}.
		\end{equation}
		Recall that, if $ f\in L ^{2} (\mathbb{R} ^{m})$ and $g\in L ^{2} (\mathbb{R} ^{n})$ , then $\Vert f\otimes _{r}g\Vert \leq \Vert f\Vert \times \Vert g\Vert$ for every $0\leq r\leq m\wedge n$. We then  obtain the following estimates, for $0\leq r\leq q-1$  and for $ 0\leq a \leq 2q-2r-2$, 
		\begin{eqnarray*}
			&&	\mathbf{E} \left(  I _{4q-4r-2a-2} \left( (g_{l, N}\tilde{\otimes} _{r}g_{l, N})\otimes _{a+1}(g_{l, N}\tilde{\otimes} _{r}g_{l, N}) \right)\right) ^{2} \\
			&=& C(q,r,a) \Vert  (g_{l, N}\tilde{\otimes} _{r}g_{l, N})\otimes _{a+1}(g_{l, N}\tilde{\otimes} _{r}g_{l, N})\Vert ^{2}\\
			& \leq& C(q,r,a) \Vert  (g_{l, N}\tilde{\otimes} _{r}g_{l, N})\Vert ^{4} \leq C(q,r,a) \Vert  (g_{l, N}\otimes _{r}g_{l, N})\Vert ^{4} \\
			&\leq & C(q,r,a) \Vert   g_{l, N} \Vert ^{8} \leq C(q,r,a) 2 ^{-8HN},
		\end{eqnarray*}
		due to (\ref{9aa-5}). Also, for $0\leq r_{1}\not=r_{2} \leq q-1$ and for $ 0\leq a \leq (2q-2r_{1}-1) \wedge ( 2q-2r_{2}-1)$, one has 
		\begin{eqnarray*}
			&& \mathbf{E} \left(  I _{4q-2r_{1}-2r_{2}-2a-2} \left( (g_{l, N}\tilde{\otimes} _{r_{1}}g_{l, N})\otimes _{a+1}(g_{l, N}\tilde{\otimes} _{r_{2}}g_{l, N}) \right)\right) ^{2} \\
			&=&C(q,r_{1}, r_{2}, a) \Vert  (g_{l, N}\tilde{\otimes} _{r_{1}}g_{l, N})\otimes _{a+1}(g_{l, N}\tilde{\otimes} _{r_{2}}g_{l, N})\Vert ^{2} \\
			&\leq & C(q,r_{1}, r_{2}, a) \Vert  (g_{l, N}\tilde{\otimes} _{r_{1}}g_{l, N})\Vert ^{2}  \Vert  (g_{l, N}\tilde{\otimes} _{r_{2}}g_{l, N})\Vert ^{2} \\
			&\leq &C(q,r_{1}, r_{2}, a) \Vert  (g_{l, N}\otimes _{r_{1}}g_{l, N})\Vert ^{2}  \Vert  (g_{l, N}\otimes _{r_{2}}g_{l, N})\Vert ^{2}\\
			&\leq & C(q,r_{1}, r_{2}, a)\Vert   g_{l, N} \Vert ^{8} \leq C(q,r_{1}, r_{2}, a)2 ^{-8HN}.
		\end{eqnarray*}
		By plugging these last two inequalities into (\ref{9aa-6}), for all $N$ large enough, we get (below we use again that $g_{l,N}$ and $g_{k,N}$ have disjoint supports, coordinate by coordinate, when $k\not=l$, also we use (\ref{iso}))
		\begin{eqnarray}
			&&	{\rm Var} \left(  \langle DV_{N,1}, D(-L) ^{-1} V_{N,1} \rangle_{\mathcal{H}} \right) \nonumber \\
			&\leq & C(q) \frac{ 2^{8HN}}{ \vert \mathcal{L}_{N, \gamma}\vert ^{2}} \sum _{l\in \mathcal{L}_{ N, \gamma}}\left[ \sup_{r,a} \mathbf{E} \left(  I _{4q-4r-2a-2} \left( (g_{l, N}\tilde{\otimes} _{r}g_{l, N})\otimes _{a+1}(g_{l, N}\tilde{\otimes} _{r}g_{l, N}) \right)\right) ^{2} \right. \nonumber \\
			&&\left. + \sup_{ r_{1}, r_{2}, a} \mathbf{E} \left(  I _{4q-2r_{1}-2r_{2}-2a-2} \left( (g_{l, N}\tilde{\otimes} _{r_{1}}g_{l, N})\otimes _{a+1}(g_{l, N}\tilde{\otimes} _{r_{2}}g_{l, N}) \right)\right) ^{2}\right]  \nonumber\\
			&\leq & C(q) \frac{1}{ \vert \mathcal{L}_{N, \gamma}\vert } \leq C(q) 2 ^{-N ^{\gamma}}, \label{9aa-8}
		\end{eqnarray}
		where the last inequality results from (\ref{8aa-7bis}). Next, we derive from Theorem \ref{tt1} and the estimates (\ref{9aa-7}) and (\ref{9aa-8}), that, for $N$ sufficiently large,  
		\begin{equation*}
			d\left( V_{N,1}, N\left(0, \mathbf{E} \vert Z_1 ^{H,q} \vert ^{4}-1\right) \right) \leq C(q) \max\left( 2 ^{N ^{\beta}\frac{2H-2}{q}}, 2 ^{-\frac{N ^{\gamma}}{2}}\right)\leq C(q) 2 ^{-\frac{ N ^{\gamma}}{2}},
		\end{equation*}
		where we used the assumption $\beta >\gamma$. \end{proof}
	
	Let us deduce the asymptotic behavior of the modified quadratic variation of the Hermite process.  By $d_{W}$ we denote the Wasserstein distance (see (\ref{dw}) for its definition). 
	
	\begin{theorem}\label{tt2}
		Let $(V_{N}, N\geq 1)$ be the sequence defined in (\ref{vn}). Then, as $N\to \infty$, 
		\begin{equation}
			\label{eq:ant3}
			V_{N} \to ^{(d)} N\left(0, \left( \mathbf{E} \vert Z^{H,q}_{1}\vert ^{4}-1\right) \right), 
		\end{equation}
		and for $N$ large enough, 
		\begin{equation}\label{3o-7}
			d_{W}\left( V_{N}, N\left(0, \mathbf{E} \vert Z_1^{H,q} \vert ^{4}-1\right) \right) \leq C(q) 2 ^{-\frac{ N ^{\gamma}}{2}}.
		\end{equation}
	\end{theorem}
	\begin{proof} The convergence in distribution follows immediately from the decomposition (\ref{deco}) and the results stated in Propositions \ref{pp3}, \ref{pp4} and \ref{pp5}. To get the rate of convergence under the Wasserstein distance, we write, via the triangle inequality and the definition of the Wasserstein distance (\ref{dw}),  
		\begin{eqnarray*}
			&&	d_{W}\left( V_{N}, N\left(0, \mathbf{E} \vert Z_1^{H,q} \vert ^{4}-1\right) \right) \\
			&\leq & 	d_{W}\left( V_{N,1}, N\left(0, \mathbf{E} \vert Z_1 ^{H,q} \vert ^{4}-1\right) \right) +d_{W}(V_{N}, V_{N,1})\\
			&\leq & d_{W}\left( V_{N,1}, N\left(0, \mathbf{E} \vert Z_1 ^{H,q} \vert ^{4}-1\right) \right)+ \mathbf{E} \vert V_{N}- V_{N,1}\vert \\
			&\leq & d_{W}\left( V_{N,1}, N\left(0, \mathbf{E} \vert Z_1 ^{H,q} \vert ^{4}-1\right) \right)+ \mathbf{E}\vert V_{N,2}\vert + \mathbf{E}\vert V_{N,3}\vert. 
		\end{eqnarray*}
		The estimates in  Propositions \ref{pp3}, \ref{pp4} and \ref{pp5} imply the desired conclusion. \end{proof}
	
	\section{Hurst parameter estimation} \label{sec5}
	
	In this section, the purpose is to estimate the Hurst parameter $H$ of the Hermite process based on the discrete observation of this process. We  only need to assume that the process $ Z^{H,q}$ is observed at the times $(e_{l,N,\beta}, e_{l, N, \beta}+2 ^{-N}, l \in \mathcal{L}_{N,\gamma})$ which is  true, in particular,  if $Z^{H,q}$ is observed on the interval $[0,1]$ at the dyadic points  $\frac{j}{2 ^{N}}$ for $j=0,1,..., 2 ^{N}$.

	Consider the sequence
	\begin{eqnarray}
		\label{def:SN}
		S_{N}&=& \frac{1}{ \vert \mathcal{L}_{N, \gamma}\vert }\sum_{l\in \mathcal{L}_{N, \gamma}} ( \Delta Z^{H,q}_{l,N}) ^{2}\nonumber \\
		&=&\frac{1}{ \vert \mathcal{L}_{N, \gamma}\vert }\sum_{l\in \mathcal{L}_{N, \gamma}}  \left( Z ^{H,q}_{ e_{l,N, \beta}+2 ^{-N}}-  Z ^{H,q}_{ e_{l,N, \beta}}\right) ^{2},
	\end{eqnarray}
	with the notation (\ref{e}). We clearly have, for every $N\geq 1$,
	\begin{equation}
		\label{9aa-88}
		\mathbf{E}S_{N}= 2 ^{-2HN}.
	\end{equation}
	In order to estimate $H$, the standard statistical procedure is to approximate $\mathbf{E}S_{N}$ by $ S_{N}$ in (\ref{def:SN}) and to take the (Napierian) logarithm in (\ref{9aa-88}). Thus, we get the estimator 
	\begin{equation}
		\label{est}
		\widehat{H}_{N} =\frac{ -\log( S_{N})}{2N\log 2 }, \hskip0.5cm N\geq 1.
	\end{equation}
	The purpose of the present section is to deduce, from the results in Section \ref{sec4}, the asymptotic properties of the above estimator. Let us start with an auxiliary result. 
	
	\begin{prop}\label{pp6} 
		For $N\geq 1$, let $ V_{N}$ be given by (\ref{vn}) and consider the set $\mathcal{L}_{N, \gamma}$ defined by (\ref{lng}). We set
		\begin{equation}
			\label{eq:ant6}
			U_{N}= 	\frac{V_{N}}{\sqrt{ \vert \mathcal{L}_{N, \gamma}\vert }  }=2^{2HN}S_N-1, \hskip0.5cm N\geq 1. 
		\end{equation}
		Then the sequence $(U_{N}, N\geq 1)$ converges almost surely to zero as $N\to \infty$ at the fast rate $2^{-N^a}$, where $a\in (0,\gamma)$ is arbitrary and fixed. 
	\end{prop}
	\begin{proof}Let $G$ be a random variable with Gaussian distribution $N(0, \mathbf{E} \vert Z_1 ^{H,q} \vert ^{4}-1)$. Using the fact that the absolute value on $\mathbb{R}$ is a Lipschitz continuous function with Lipschitz semi-norm $\| \cdot \|_{Lip}$ equals to $1$, we can derive from the definition of the Wasserstein distance $d_W$ (see (\ref{dw})) and from Theorem 2, that, for every $N$ large enough, we have
		\[
		\mathbf{E}|V_N|-\mathbf{E}|G|\le\big |\mathbf{E}|V_N|-\mathbf{E}|G|\big|\le d_W (V_N, G)\le  C(q) 2 ^{-\frac{1}{2}N ^{\gamma}}.
		\] 
		Then, combining (\ref{eq:ant6}) and (\ref{8aa-7bis}), we get, for all $N\ge 1$,
		\begin{equation}
			\label{eq:ant1}
			\mathbf{E}|U_N|\le C_1(q) 2^{-2^{-1}N^\gamma}.
		\end{equation}
		Next, let $a\in (0,\gamma)$ be arbitrary and fixed. We can derive from (\ref{eq:ant1}) and Markov's inequality that
		\begin{eqnarray*}
			\sum_{N\geq 1} P \left( \vert U_{N} \vert \geq 2 ^{-N ^{a}}\right) \leq \sum_{N\geq 1} 2 ^{N ^{a}}\mathbf{E} \vert U_{N} \vert 
			\leq C \sum_{N\geq 1} 2 ^{(N ^{a}-2^{-1}N ^{\gamma})} <\infty .
		\end{eqnarray*}
		Then, the almost sure convergence of the sequence $ (U_{N}, N\geq 1)$ to zero at the fast rate $2^{-N^a}$ follows by Borel-Cantelli's lemma. \end{proof}

	We have the following result.

	\begin{theorem}\label{tt3}
		The estimator (\ref{est}) is strongly consistent, i.e. $\widehat{H} _{N} \to _{N \to \infty}H$ almost surely, and its almost sure convergence holds at the fast rate $2^{-N^a}$, where $a\in (0,\gamma)$ is arbitrary and fixed.  Moreover,  
		it is asymptotically normal:
		
		\begin{equation}\label{3o-6}
			2N(\log 2) \sqrt{ \vert \mathcal{L}_{N, \gamma}\vert }\left(H-\widehat{H}_{N}\right)\to ^{(d)} _{N \to \infty} N(0, \mathbf{E} \vert Z_1^{H,q}\vert ^{4}-1).
		\end{equation}
	\end{theorem}
	\begin{proof}  Let us first precisely determine the relationship between $\widehat{H}_{N}$ and the modified quadratic variation $V_{N}$. From (\ref{vn}), (\ref{def:SN}) and (\ref{eq:ant6}), we can write
		\begin{equation*}
			V_{N}= \sqrt{ \vert \mathcal{L}_{N, \gamma}\vert } \left( 2 ^{2HN}S_{N}-1\right),
		\end{equation*}
		so
		\begin{equation*}\label{10aa-1}
			\frac{V_{N}}{\sqrt{ \vert \mathcal{L}_{N, \gamma}\vert }  }+1=U_{N}+1=2 ^{2HN}S_{N},
		\end{equation*}
		or, equivalently,
		\begin{equation}\label{22n-1}
			2N\log 2 \sqrt{ \vert \mathcal{L}_{N, \gamma}\vert }\left( H-\widehat{H}_{N}\right) = V_{N}+  \sqrt{ \vert \mathcal{L}_{N, \gamma}\vert }\left( \log(1+U_{N})-U_{N}\right).
		\end{equation}
		Then, we can derive from (\ref{22n-1}) and the inequality 
		\begin{equation}
			\label{22n-2}
			\big | \log (1+x)-x \big|\le x^2,\quad \mbox{for all $x\in [-1/2,1/2]$,}
		\end{equation}
		and the fact that that $U_N=\frac{V_{N}}{\sqrt{ \vert \mathcal{L}_{N, \gamma}\vert }  }\to _{N\to \infty}0$ almost surely (see Proposition \ref{pp6}), that we have almost surely, for each $N$ large enough,
		\[
		\big | H-\widehat{H}_N\big | \le (2 \log 2)^{-1}\, N^{-1}\big ( \vert U_N\vert +U_N^2\big).
		\]
		Thus, in view of Proposition \ref{pp6}, it turns out that the estimator $\widehat{H}_N$ converges almost surely to the Hurst parameter $H$ 
		at the fast rate $2^{-N^a}$, where $a\in (0,\gamma)$ is arbitrary and fixed.
		
		Let us now show that (\ref{3o-6}) holds.  In view of (\ref{22n-1}) and Theorem \ref{tt2} it is enough to prove 
		that
		\begin{equation}\label{22n-3}
			\sqrt{ \vert \mathcal{L}_{N, \gamma}\vert }\left( \log(1+U_{N})-U_{N}\right)\to _{N\to \infty} 0 \mbox{ almost surely.}
		\end{equation}
		It follows from Proposition \ref{pp6}, (\ref{22n-2}), and (\ref{eq:ant6}), that we have, almost surely, for all $N$ large enough,
		\[
		\sqrt{ \vert \mathcal{L}_{N, \gamma}\vert }\left |\log(1+U_{N})-U_{N}\right| \le \frac{V_{N}^2}{\sqrt{ \vert \mathcal{L}_{N, \gamma}\vert }  }.
		\]
		Thus, (\ref{22n-3}) can be obtained by showing that 
		\begin{equation}\label{ant12}
			\frac{V_{N}^2}{\sqrt{ \vert \mathcal{L}_{N, \gamma}\vert }  }\to _{N\to \infty} 0 \mbox{ almost surely.}
		\end{equation}
		Observe that Remark \ref{rem:bvn} and (\ref{8aa-7bis}) imply that 
		\begin{equation}\label{ant13}
			\mathbf{E}\left (\frac{V_{N}^2}{\sqrt{ \vert \mathcal{L}_{N, \gamma}\vert }  }\right)=\mathcal{O}\left (2^{-2^{-1} N^\gamma}\right),\quad\mbox{for all $N$ large enough.}
		\end{equation}
		Finally, we can derive from (\ref{ant13}), Markov's inequality and Borel-Cantelli's lemma that (\ref{ant12}) is satisfied.
		
	\end{proof}
	
	Before ending this section let us discuss the roles of the two parameters $\beta\in (0,1)$ and $\gamma\in (0,\beta)$ which were introduced in (\ref{ln}) and (\ref{lng}).
	\begin{remark}
		\label{rem:rolebega}
		As we have already emphasized, in our present work we use for statistical inference the modified quadratic variation $V_N$ (see (\ref{vn})), which is obtained through some well-chosen dyadic increments of the Hermite process $Z^{H,q}$ and not all its dyadic increments. The parameter $\beta\in (0,1)$ is related with the selection procedure of these nice increments. Namely, among all the dyadic increments of the Hermite process $Z ^{H,q}_{(k+1)/2 ^{N}}- Z ^{H,q}_{k/2 ^{N}}$, $0\le k <2^N$, we only select those for which the integer $k$ is multiple of $[2 ^{N ^{\beta}}]$. Thus, when $\beta$ decreases and gets closer to zero the number of the selected increments increases; which somehow means that more data is available for statistical inference through our modified quadratic variation $V_N$. Yet, if we take in it all the selected increments for a given $\beta$ (i.e. if we replace in (\ref{vn}) the set $\mathcal{L}_{N, \gamma}$ by the set $\mathcal{L}_{N}$), then our proof of the Central Limit Theorem for $V_N$ (see Theorem \ref{tt2}) will no longer work. For this reason, we need to take in $V_N$ about $[2 ^{N ^{\gamma}}]$ of the selected increments, with $\gamma\in (0,\beta)$ being another parameter which should be chosen as close as possible to $\beta$, for improving as far as possible, for a given $\beta$, the rate of convergence toward normal distribution, provided by (\ref{3o-7}) in Theorem \ref{tt2}. It is an open difficult question to know what is the "best" choice of the parameter $\beta$. On one hand, it is tempting to seek to increase the value of $\beta$ since the chosen value for $\gamma$ can then be larger and thus the rate of convergence of $V_N$ toward normal distribution becomes better; on the other hand increasing the value of $\beta$ somehow means that less data is available for statistical inference, which may be a drawback in one way or the other. 
	\end{remark}

	\section{Quadratic variation and Hurst estimation for the Hermite-Ornstein-Uhlenbeck process}
	\label{sec6}
	The results obtained in Sections \ref{sec4} and \ref{sec5} can be easily transferred to other Hermite-related stochastic processes. To illustrate this point, we study in the present section the case of the Ornstein-Uhlenbeck process associated to the Hermite process, which will be called the Hermite Ornstein-Uhlenbeck (HOU) process in the sequel. 
	
	The HOU process is defined as the unique solution of the Langevin equation 
	\begin{equation}\label{hou}
		X_{t}= \xi-\int_{0} ^{t} X_{s}ds + Z^{H,q} _{t},  \hskip0.5cm t\geq 0
	\end{equation}
	with initial condition $\xi \in L ^ {2} (\Omega)$, where $ Z ^{H,q}$ is a Hermite process of  any order $q\geq 1$ with  an arbitrary self-similarity index $H\in (\frac{1}{2}, 1)$.  For simplicity, we take $\xi=0$.   We mention in passing that the solution of the Langevin equation (\ref{hou}) can be expressed as a Wiener-integral with respect to $ Z^{H,q}$ (see \cite{AT} or \cite{NT}), but we will not need to use this integral representation in the present section.  For more details on the HOU process, we refer to \cite{AT}, \cite{CKM} or \cite{NT}. For later purposes, we recall its following property: for every $p\geq 2$ and $T>0$,
	\begin{equation}\label{26n-4}
		\sup_{t\in [0, T]} \mathbf{E} \vert X_{t}\vert ^ {p} < C,
	\end{equation}
	where $C$ is a positive finite constant depending only on $p$ and $T$. 
	
	For every integer $N\geq 1$, the modified quadratic variation $V_N (X)$ is defined as
	\begin{eqnarray}
		V_{N}(X) &=& \frac{ 2 ^{2HN}}{ \sqrt{ \vert \mathcal{L}_{N, \gamma}\vert}}\sum_{l\in \mathcal{L}_{N, \gamma}} \left[ \left(X_{ e_{l,N, \beta}+2 ^{-N}}-  X_{ e_{l,N, \beta}}\right) ^{2}- \mathbf{E} \left( Z ^{H,q}_{ e_{l,N, \beta}+2 ^{-N}}-  Z ^{H,q}_{ e_{l,N, \beta}}\right) ^{2}\right]\nonumber \\
		&=&\frac{ 2 ^{2HN}}{ \sqrt{ \vert \mathcal{L}_{N, \gamma}\vert}}\sum_{l\in \mathcal{L}_{N, \gamma}} \left[( \Delta X_{l,N}) ^{2}- \mathbf{E}( \Delta Z^{H,q}_{l,N} )^{2}\right],\label{vnx}
	\end{eqnarray}
	where $ \mathcal{L}_{N, \gamma}$ is the same set as in (\ref{lng}) and $e_{l,N, \beta}$ is given by (\ref{e}). Also, we mention that
	\begin{equation*}
		\Delta X_{l,N} =X_{ e_{l,N, \beta}+2 ^{-N}}-  X_{ e_{l,N, \beta}}.
	\end{equation*}
	As a consequence of Theorem \ref{tt2}, we deduce the following asymptotic behavior in distribution of $V_{N}(X)$:
	\begin{prop}\label{pp7}
		Let $ (V_{N}(X), N\geq 1)$ be given by (\ref{vnx}).  As $N\to \infty$, 
		\begin{equation*}
			V_{N}(X) \to ^{(d)} N\left(0, \left( \mathbf{E} \vert Z^{H,q}_{1}\vert ^{4}-1\right) \right), 
		\end{equation*}
		and for $N$ large enough, 
		\begin{equation*}
			d_{W}\left( V_{N}(X), N\left(0, \mathbf{E} \vert Z_1 ^{H,q} \vert ^{4}-1\right) \right) \leq C(q) 2 ^{-\frac{N ^{\gamma}}{2}}.
		\end{equation*}
	\end{prop}
	\begin{proof}For $t\geq 0$, let 
		\begin{equation*}
			Y_{t}= -\int_{0} ^{t} X_{s}ds,
		\end{equation*}
		where $X$ is the solution to (\ref{hou}). Then we can derive from (\ref{hou}) and elementary computations that
		\begin{equation}
			\label{decoVNX}
			V_{N}(X)= V_{N}+ \frac{ 2 ^{2HN}}{ \sqrt{ \vert \mathcal{L}_{N, \gamma}\vert}}\sum_{l\in \mathcal{L}_{N, \gamma}} ( \Delta Y_{l,N}) ^{2}+2 \frac{ 2 ^{2HN}}{ \sqrt{ \vert \mathcal{L}_{N, \gamma}\vert}}\sum_{l\in \mathcal{L}_{N, \gamma}} (\Delta Z^{H,q}_{l,N})  ( \Delta Y_{l,N}),
		\end{equation}
		where $V_{N}$ is given by (\ref{vn}) and $\Delta Y_{l,N} =Y_{ e_{l,N, \beta}+2 ^{-N}}-  Y_{ e_{l,N, \beta}}.$ Notice that, in view of Theorem \ref{tt2}, when $N$ goes to $\infty$, the first term in the right-hand side of (\ref{decoVNX}) converges in distribution to the desired limit at the desired rate $2 ^{-\frac{N ^{\gamma}}{2}}$. Thus, in order to obtain the proposition, it is enough to show that, when $N$ tends to $\infty$,  the other two terms, in the right-hand side of (\ref{decoVNX}), converge to zero in $ L^{1}(\Omega)$ at a faster rate than $2 ^{-\frac{N ^{\gamma}}{2}}$. This is true. Indeed, using Cauchy-Schwarz's inequality as well as the inequalities (\ref{26n-4}) and (\ref{8aa-7}), we get that
		\begin{eqnarray*}
			&&\mathbf{E} \left|  \frac{ 2 ^{2HN}}{ \sqrt{ \vert \mathcal{L}_{N, \gamma}\vert}}\sum_{l\in \mathcal{L}_{N, \gamma}} ( \Delta Y_{l,N}) ^{2}\right| =  \frac{ 2 ^{2HN}}{ \sqrt{ \vert \mathcal{L}_{N, \gamma}\vert}}\sum_{l\in \mathcal{L}_{N, \gamma}} \mathbf{E} \left( \int_{e_{l,N,\beta}}^{e_{l,N, \beta}+ 2 ^{-N}}X_{s}ds \right) ^{2} \\
			&&\leq  \frac{ 2 ^{2HN}}{ \sqrt{ \vert \mathcal{L}_{N, \gamma}\vert}}2 ^{-N}\sum_{l\in \mathcal{L}_{N, \gamma}} \int_{e_{l,N,\beta}}^{e_{l,N, \beta}+ 2 ^{-N}}\mathbf{E} X_{s}^{2} ds \leq C 2 ^{(2H-2)N} \sqrt{ \vert \mathcal{L}_{N, \gamma}\vert}\\
			&&\leq C 2 ^{(2H-2)N + \frac{N ^{\gamma}}{2}}=o\left (2^{-(1-H)N}\right),\quad\mbox{for all $N$ large enough,}
		\end{eqnarray*}
		where the last equality results from the fact that $\gamma <1$. Similar arguments allow us to derive that
		\begin{eqnarray*}
			&&\mathbf{E}\left|  \frac{ 2 ^{2HN}}{ \sqrt{ \vert \mathcal{L}_{N, \gamma}\vert}}\sum_{l\in \mathcal{L}_{N, \gamma}} (\Delta Z^{H,q}_{l,N})  ( \Delta Y_{l,N}) \right| \leq  \frac{ 2 ^{2HN}}{ \sqrt{ \vert \mathcal{L}_{N, \gamma}\vert}}\sum_{l\in \mathcal{L}_{N, \gamma}} \left( \mathbf{E}(\Delta Z^{H,q}_{l,N}) ^{2} \right) ^{\frac{1}{2}} \left( \mathbf{E} ( \Delta Y_{l,N})^{2} \right) ^{\frac{1}{2}} \\
			&&\leq C 2 ^{(H-1)N + \frac{N ^{\gamma}}{2}}=o\left (2^{-(1-H)\frac{N}{2}}\right),\quad\mbox{for all $N$ large enough.} 
		\end{eqnarray*}
		
	\end{proof}
	
	In view of Proposition \ref{pp7}, the same procedure as in Section \ref{sec5} can be used in order to obtain a strongly consistent and asymptotically normal estimator for the Hurst parameter $H$ in (\ref{hou}). More precisely, if $X$ is the HOU process defined by (\ref{hou}), let
	
	\begin{equation*}
		S_{N}(X)= \frac{1}{ \vert \mathcal{L}_{N, \gamma}\vert }\sum_{l\in \mathcal{L}_{N, \gamma}} ( \Delta X_{l,N}) ^{2}, \hskip0.5cm N\geq 1.
	\end{equation*}
	Similarly to (\ref{decoVNX}), it can be shown that
	\[
	2^{2HN}S_{N}(X)= 2^{2HN}S_{N}+ \frac{ 2 ^{2HN}}{ \vert \mathcal{L}_{N, \gamma}\vert}\sum_{l\in \mathcal{L}_{N, \gamma}} ( \Delta Y_{l,N}) ^{2}+2 \frac{ 2 ^{2HN}}{\vert \mathcal{L}_{N, \gamma}\vert}\sum_{l\in \mathcal{L}_{N, \gamma}} (\Delta Z^{H,q}_{l,N})  ( \Delta Y_{l,N}),
	\]
	where $S_N$ is as in (\ref{def:SN}). Then, 
	we can derive from Proposition \ref{pp6}, Borel-Cantelli's lemma and the estimates obtained in the proof of Proposition \ref{pp7} that we have $ S_{N}(X) \sim 2 ^{-2HN}$ (in the sense  that $2^{2HN}S_{N}(X)-1$ converges almost surely to zero as $N\to \infty$). Thus, it turns out that 
	\begin{equation*}
		\widehat{H}_{N}(X) =\frac{ -\log( S_{N}(X))}{2N\log 2 }
	\end{equation*}
	is a strongly consistent estimator for the Hurst parameter $H$ of the Hermite Ornstein-Uhlenbeck process. Moreover, by using Proposition \ref{pp7} and by arguing as in the proof of Theorem \ref{tt3}, we can show that this estimator is asymptotically normal:
	\begin{equation*}
		2N(\log 2) \sqrt{ \vert \mathcal{L}_{N, \gamma}\vert }\left( H-\widehat{H}_{N}(X)\right)\to ^{(d)} _{N \to \infty} N(0, \mathbf{E} \vert Z_1^{H,q}\vert ^{4}-1).
	\end{equation*}

	\section{Appendix: Multiple stochastic integrals and the Malliavin derivative}\label{app}
	
	The basic tools from the analysis on Wiener space are presented in this section. We will focus on some elementary facts about multiple stochastic integrals. We refer to \cite{N} or \cite{NP-book} for a complete review on the topic. 
	
	Consider ${\mathcal{H}}$ a real separable infinite-dimensional Hilbert space
	with its associated inner product ${\langle
		\cdot,\cdot\rangle}_{\mathcal{H}}$, and $(B (\varphi),
	\varphi\in{\mathcal{H}})$ an isonormal Gaussian process on a
	probability space $(\Omega, {\mathfrak{F}}, \mathbb{P})$, which is a
	centered Gaussian family of random variables such that
	$\mathbf{E}\left( B(\varphi) B(\psi) \right) = {\langle\varphi,
		\psi\rangle}_{{\mathcal{H}}}$ for every
	$\varphi,\psi\in{\mathcal{H}}$. Denote by $I_{q} (q\geq 1)$ the $q$th multiple
	stochastic integral with respect to $B$, which is an
	isometry between the Hilbert space ${\mathcal{H}}^{\odot q}$
	(symmetric tensor product) equipped with the scaled norm
	$\sqrt{q!}\,\Vert\cdot\Vert_{{\mathcal{H}}^{\otimes q}}$ and
	the Wiener chaos of order $q$, which is defined as the closed linear
	span of the random variables $H_{q}(B(\varphi))$ where
	$\varphi\in{\mathcal{H}},\;\Vert\varphi\Vert_{{\mathcal{H}}}=1$ and
	$H_{q}$ is the Hermite polynomial of degree $q\geq 1$ defined
	by :\begin{equation}\label{Hermite-poly}
		H_{q}(x)=\frac{(-1)^{q}}{q!} \exp \left( \frac{x^{2}}{2} \right) \frac{{\mathrm{d}}^{q}%
		}{{\mathrm{d}x}^{q}}\left( \exp \left(
		-\frac{x^{2}}{2}\right)\right),\;x\in \mathbb{R}.
	\end{equation}
	For $q=0$, 
	\begin{equation}\label{27a-1}
		\mathcal{H}_{0}=\mathbb{R} \mbox{ and }I_{0}(x)=x \mbox{ for every }x\in \mathbb{R}.
	\end{equation}
	The isometry property of multiple integrals can be written as follows : for $p,\;q\geq
	0$,\;$f\in{{\mathcal{H}}^{\otimes p}}$ and
	$g\in{{\mathcal{H}}^{\otimes q}}$
	\begin{equation} \mathbf{E}\Big(I_{p}(f) I_{q}(g) \Big)= \left\{
		\begin{array}{rcl}\label{iso}
			q! \langle \tilde{f},\tilde{g}
			\rangle _{{\mathcal{H}}^{\otimes q}}&&\mbox{if}\;p=q,\\
			\noalign{\vskip 2mm} 0 \quad\quad&&\mbox{otherwise,}
		\end{array}\right.
	\end{equation}
	where $\tilde{f}$ stands for the symmetrization of $f$. When $\mathcal{H}= L^{2}(T)$, with $T$ being an interval of $\mathbb{R}$, we have the following product formula: for $p,\;q\geq
	0$,\;  $f\in{{\mathcal{H}}^{\odot p}}$ and
	$g\in{{\mathcal{H}}^{\odot q}}$,  
	
	\begin{eqnarray}\label{prod}
		I_{p}(f) I_{q}(g)&=& \sum_{r=0}^{p \wedge q} r! \binom{q}{r} \binom{p}{r}I_{p+q-2r}\left(f\otimes_{r}g\right),
	\end{eqnarray}
	where, for $r=0, ..., p\wedge q$, the contraction $f\otimes _{r} g$ is the function in $L ^{2}( T ^{p+q-2r}) $ given by 
	
	\begin{equation}\label{contra}
		(f\otimes _{r} g) (t_{1},..., t_{p+q-2r})= \int_{T^{r}} f(u_{1},..., u_{r}, t_{1},..., t_{p-r}) g(u_{1},..., u_{r}, t_{p-r+1},...,t_{p+q-2r}) du_{1}...du_{r}.
	\end{equation}

	An useful  property of  finite sums of multiple stochastic integrals is the hypercontractivity. Namely,  for every fixed real number $p\geq 2$, there exists a universal deterministic finite constant $C_p$, such that, for any random variable $F$ of the form $F= \sum_{k=0} ^{n} I_{k}(f_{k}) $ with $f_{k}\in \mathcal{H} ^{\otimes k}$, the following inequality holds:
	\begin{equation}
		\label{hyper}
		\mathbf{E}\vert F \vert ^{p} \leq C_{p} \left( \mathbf{E}F ^{2} \right) ^{\frac{p}{2}}.
	\end{equation}

	We denote by $D$ the Malliavin derivative operator that acts on
	cylindrical random variables of the form $F=g(B(\varphi
	_{1}),\ldots,B(\varphi_{n}))$, where $n\geq 1$,
	$g:\mathbb{R}^n\rightarrow\mathbb{R}$ is a smooth function with
	compact support and $\varphi_{i} \in {{\mathcal{H}}}$, in the following way:
	\begin{equation*}
		DF=\sum_{i=1}^{n}\frac{\partial g}{\partial x_{i}}(B(\varphi _{1}),
		\ldots , B(\varphi_{n}))\varphi_{i}.
	\end{equation*}
	The operator $D$ is closable and it can be extended to $\mathbb{D} ^{1, 2}$ which denotes the closure of the set of cylindrical random variables with respect to the norm $\Vert \cdot\Vert _{1,2}$ defined as
	\begin{equation*}
		\Vert F\Vert _{1,2} ^{2}:= \mathbf{E}\vert F\vert ^{2}+ \mathbf{E} \Vert  DF\Vert _{\mathcal{H}} ^{2}. 
	\end{equation*} 
	If $F=I_{p}(f)$, where $f\in \mathcal{H} ^{\odot p}$ with $\mathcal{H}= L^{2}(T)$ and $p\geq 1$, then 
	$$D_{\ast}F=pI_{p-1} \left( f(\cdot, \ast)\right),$$
	where $"\cdot "$ stands for $p-1$ variables.
	
	The pseudo inverse $ (-L) ^{-1}$ of the Ornstein-Uhlenbeck operator $L$ is defined, for $F=I_{p}(f)$ with $f\in \mathcal{H} ^{\odot p}$ and $p\geq 1$, by 
	\begin{equation*}
		(-L) ^{-1}F= \frac{1}{p} I_{p} (f).
	\end{equation*}
	
	At last notice that in our work, we have $\mathcal{H}= L ^{2} (\mathbb{R})$ while the role of the isonormal process $(B(\varphi), \varphi \in \mathcal{H}) $ is played by  the usual Wiener integral on $L ^{2} (\mathbb{R})$ associated with the Brownian motion $(B(y), y \in \mathbb{R})$.
	
	%%%%%%%%%%%%%%%%%%%%%%%%%%%%%%%%%%%%%%%%%%%%%%
	%%%% Main text entry area:

	%%%%%%%%%%%%%%%%%%%%%%%%%%%%%%%%%%%%%%%%%%%%%%
	%% Single Appendix:                         %%
	%%%%%%%%%%%%%%%%%%%%%%%%%%%%%%%%%%%%%%%%%%%%%%
	%\begin{appendix}
	%\section*{???}%% if no title is needed, leave empty \section*{}.
	%\end{appendix}
	%%%%%%%%%%%%%%%%%%%%%%%%%%%%%%%%%%%%%%%%%%%%%%
	%% Multiple Appendixes:                     %%
	%%%%%%%%%%%%%%%%%%%%%%%%%%%%%%%%%%%%%%%%%%%%%%
	%\begin{appendix}
	%\section{???}
	%
	%\section{???}
	%
	%\end{appendix}
	
	%%%%%%%%%%%%%%%%%%%%%%%%%%%%%%%%%%%%%%%%%%%%%%
	%% Support information, if any,             %%
	%% should be provided in the                %%
	%% Acknowledgements section.                %%
	%%%%%%%%%%%%%%%%%%%%%%%%%%%%%%%%%%%%%%%%%%%%%%

	%%%%%%%%%%%%%%%%%%%%%%%%%%%%%%%%%%%%%%%%%%%%%%
	%% Funding information, if any,             %%
	%% should be provided in the                %%
	%% funding section.                         %%
	%%%%%%%%%%%%%%%%%%%%%%%%%%%%%%%%%%%%%%%%%%%%%%
\vskip0.5cm

\noindent{\bf Acknowledgments: }
		The  authors acknowledge partial support from the Labex CEMPI (ANR-11-LABX-007-01) and the GDR 3475 (Analyse Multifractale et Autosimilarit\'e).
		A. Ayache also acknowledges partial support from the Australian Research Council's Discovery Projects funding scheme (project number  DP220101680).  C. Tudor also acknowledges partial support from the projects ANR-22-CE40-0015, MATHAMSUD (22-
		MATH-08),  ECOS SUD (C2107), Japan Science and Technology Agency CREST  JPMJCR2115 and  by a grant of the Ministry of Research, Innovation and Digitalization (Romania), CNCS-UEFISCDI, PN-III-P4-PCE-2021-0921, within PNCDI III.


\begin{thebibliography}{99}
		
		
		\bibitem{AT}
		Assaad, O. and Tudor, C.A. (2020).
		Parameter identification for the Hermite Ornstein-Uhlenbeck process. \textit{Stat. Inference Stoch. Process.}
		\textbf{23} 251-270.
	
		
		
		
		\bibitem{A}
		Ayache, A. (2020).
		Lower bound for local oscillations of Hermite processes. \textit{Stochastic Process. Appl.}
		\textbf{130} 4593–4607. 
	
		
		
		\bibitem{BaiTa2}
		Bai, S. and Taqqu, M.S. (2020).
		Limit theorems for long-memory flows on Wiener chaos. \textit{Bernoulli}
		\textbf{26} 1473–1503. 
	
		
		\bibitem{BaTu}
		Bardet J-M. and Tudor, C.A. (2010).
		A wavelet analysis of the Rosenblatt process: chaos expansion and estimation of the self-similarity parameter. \textit{Stochastic Process. Appl.}
		\textbf{120} 2331–2362. 
	
		
		
		\bibitem{BDT}
		Bourguin, S., Diez, C-P. and Tudor, C.A. (2021).
		Limiting behavior of large correlated Wishart matrices with chaotic entries. \textit{Bernoulli}
		\textbf{27} 1077–1102.
		
		
		%\bibitem[\protect\citeauthoryear{Bubeck, S., Ding, J., Eldan, R. and Racz, M.}{2016}]{bubeck_testing_2016}
		%	Bubeck, S., Ding, J., Eldan, R. and Racz, M. (2016)
		%	Testing for high-dimensional geometry in random graphs. \textit{Random Structures Algorithms},
		%	\textbf{49}, 503--532.
		%	\MR{3545825}
		
		
		\bibitem{CKM}
		Cheridito, P., Kawaguchi, H. and Maejima, M. (2003).
		Fractional Ornstein-Uhlenbeck processes. \textit{Electron. J. Probab.}
		\textbf{8} 1-14.
	
		
		
		
		\bibitem{CTV2}
		Chronopoulou, A., Tudor, C.A. and Viens, F. G. (2012).
		Self-similarity
		parameter estimation and reproduction property for non-Gaussian
		Hermite processes. \textit{Commun. Stoch. Anal.}
		\textbf{5} 161-185.
	
		
		
		\bibitem{CTV}
		Chronopoulou, A., Tudor, C.A. and Viens, F. G. (2009).
		Variations and Hurst index estimation for a Rosenblatt process using longer filters. \textit{Electron. J. Statist.}
		\textbf{3} 1393--1435.
	
		
		
		
		\bibitem{CRTT1}
		Clausel, M., Roueff, F., Taqqu, M.S. and Tudor, C.A. (2013).
		High order chaotic limits of wavelet scalograms under long-range dependence. \textit{ALEA Lat. Am. J. Probab. Math. Stat.}
		\textbf{10} 979–1011.
	
		
		\bibitem{CRTT2}
		Clausel, M., Roueff, F., Taqqu, M.S. and Tudor, C.A. (2014).
		Wavelet estimation of the long memory parameter for Hermite polynomial of Gaussian processes. \textit{ESAIM Probab. Stat.}
		\textbf{18} 42–76.
		
		
		
		\bibitem{coeur}
		Coeurjolly, J-F. (2001).
		Estimating the parameters of a
		fractional Brownian motion by discrete variations of its sample
		paths. \textit{Stat. Inference Stoch. Process.}
		\textbf{30} 199-227.  
	
		
		
		\bibitem{DoMa79}
		Dobrushin, R.L. and Major, P. (1979).
		Non-central limit theorems for non-linear functional of Gaussian
		fields. \textit{Z. Wahrscheinlichkeitstheor. Verwandte
			Geb.}
		\textbf{50} 22–52. 	
	
		
		
		\bibitem{IL97}
		Istas, J. and Lang, G. (1997).
		Quadratic variations and estimation of the local Hölder index of a Gaussian process. \textit{Ann. Inst. H. Poincaré Probab. Statist.}
		\textbf{33} 407–436.
	
		
		
		\bibitem{jan97}
		Janson, S. (1997). \textit{ Gaussian {H}ilbert spaces}. Cambridge: Cambridge University Press.

		
		
		
		\bibitem{NP-book}
		Nourdin, I. and Peccati, G. (2012). \textit{ Normal Approximations with
			Malliavin Calculus From Stein's Method to Universality}. Cambridge: Cambridge University Press.
	
		
		
		
		\bibitem{NT}
		Nourdin, I. and Tran, D. (2019).
		Statistical inference for Vasicek-type model driven by Hermite processes. \textit{Stochastic Process. Appl. } \textbf{129}  3774–3791. 
	
		
		
		
		
		
		\bibitem{N}
		Nualart, D. (2006). \textit{ Malliavin Calculus and Related Topics. Second Edition}. New York: Springer. 
	
		
		
		\bibitem{PiTa-book}
		Pipiras, V. and Taqqu, M.S. (2017). \textit{ Long-range dependence and self-similarity}. Cambridge: Cambridge University Press.
	
		
		
		\bibitem{Taq75}
		Taqqu, M.S. (1975).
		Weak convergence to fractional {B}rownian motion and to the
		{R}osenblatt process. \textit{Probab.Theory Related Fields.}
		\textbf{31} 287–302.
	
		
		
		
		\bibitem{T}
		Tudor, C.A. (2013). \textit{Analysis of variations for self-similar
			processes. }  Cham: Springer.
	
		
		
		\bibitem{TV}
		Tudor, C.A. and Viens, F.G. (2009).
		Variations and estimators for
		self-similarity parameters via {Malliavin} calculus. \textit{Ann. Probab.}
		\textbf{37} 2093--2134.
	
		
		
	\end{thebibliography}
\end{document}